\documentclass[12pt]{amsart}
\usepackage{amsfonts,amssymb, amscd, latexsym, graphicx, psfrag, color,float}
\usepackage{bbm}

\usepackage[all]{xy}

\usepackage{booktabs,enumerate}
\usepackage{multirow}
\usepackage{mathrsfs}

\usepackage{tikz}
\usetikzlibrary{matrix}
\usetikzlibrary{shapes.geometric}

\usepackage[margin=1in,marginparwidth=0.8in, marginparsep=0.1in]{geometry}

\usepackage[bookmarks=true, bookmarksopen=true,%
bookmarksdepth=3,bookmarksopenlevel=2,%
colorlinks=true,%
linkcolor=blue,%
citecolor=blue,%
filecolor=blue,%
menucolor=blue,%
urlcolor=blue]{hyperref}

\newtheorem{dummy}{dummy}[section]

\newtheorem{theorem}[dummy]{Theorem}

\newtheorem{proposition}[dummy]{Proposition}
\theoremstyle{definition}
\newtheorem{definition}[dummy]{Definition}
\newtheorem{construction}[dummy]{Construction}
\newtheorem{example}[dummy]{Example}
\newtheorem{remark}[dummy]{Remark}

\newcommand{\bC}{\mathbf{C}}

\newcommand{\bF}{\mathbf{F}}
\newcommand{\bG}{\mathbf{G}}

\newcommand{\bP}{\mathbf{P}}

\newcommand{\bR}{\mathbf{R}}
\newcommand{\bZ}{\mathbf{Z}}

\newcommand{\cC}{\mathcal{C}}

\newcommand{\cM}{\mathcal{M}}
\newcommand{\cN}{\mathcal{N}}
\newcommand{\cP}{\mathcal{P}}

\newcommand{\cT}{\mathcal{T}}

\newcommand{\op}{\operatorname}

\newcommand{\PGL}{\mathrm{PGL}}

\newcommand{\Ext}{\mathrm{Ext}}

\newcommand{\Hom}{\mathrm{Hom}}

\renewcommand{\log}{{\op{log}}}

\newcommand{\Sh}{\mathit{Sh}}

\newcommand{\Gm}{\mathbb{G}_{\mathrm{m}}}

\newcommand{\Li}{\mathrm{Li}}

\newcommand{\Mfr}{\cM_{\mathit{fr}}}
 \numberwithin{equation}{subsection}
\numberwithin{figure}{subsection}

\newcommand{\leg}{S}

\newcommand{\red}[1]{{\color{red}#1}}

\usepackage{tikz}

\title[Legendrian Surfaces]{{Cubic Planar Graphs and Legendrian Surface Theory}
} 
\author[David Treumann and Eric Zaslow]{David Treumann${}^*$ and Eric Zaslow${}^{**}$\\
{}\\
{\tiny * Department of Mathematics, Boston College}\\
{\tiny ** Department of Mathematics, Northwestern University}}

\begin{document}

\maketitle
\begin{abstract}
We study Legendrian surfaces determined by cubic planar graphs.  Graphs with distinct chromatic polynomials determine surfaces that are not Legendrian isotopic, thus giving many examples of non-isotopic Legendrian surfaces with the same classical invariants.  The Legendrians have no exact Lagrangian fillings, but have many interesting non-exact fillings.

We obtain these results by studying sheaves on a three-ball with microsupport in the surface.  The moduli of such sheaves has a concrete description in terms of the graph and a beautiful embedding as a holomorphic Lagrangian submanifold of a symplectic period domain, a Lagrangian that has appeared in the work of
Dimofte-Gabella-Goncharov \cite{DGGo}.  We exploit this
structure to find conjectural open Gromov-Witten invariants for the non-exact filling, following Aganagic-Vafa \cite{AV,AVms}.

\end{abstract}

\setcounter{tocdepth}{1}
\tableofcontents

\section{Introduction and Summary}
\label{sec:intro}

An exact Lagrangian filling of a Legendrian in a cosphere bundle determines a family of constructible sheaves \cite{NZ}.  
In this paper, we explore a curious counterpoint: Legendrian surfaces that give rise to beautiful moduli spaces of constructible sheaves, but have no exact fillings whatsoever.

For one-dimensional Legendrians, the families of fillings give the whole moduli space of constructible sheaves the rich structure of a cluster variety. 
This observation leads
to strong new lower bounds on the number of Hamiltonian isotopy classes of exact Lagrangian surfaces filling Legendrian
knots \cite{STW,STWZ}.  It is natural to wonder what structures are determined by Legendrian surfaces.

A fundamental example is related to the \emph{Harvey-Lawson cone}, a singular exact Lagrangian in $\bR^6$.  Nadler studied the microlocal category of this cone in \cite{N-cone}, proving that it is equivalent to a category of constructible sheaves on $\bR^3$ and furthermore equivalent to the category of coherent sheaves on
the pair of pants $\bP^1 \setminus \{0,1,\infty\}$.  We observe in this paper that this implies the non-fillability of the Legendrian boundary of the Harvey-Lawson cone. 
Nadler's example is fundamental for us:  we prove similar
results for a broad class of Legendrian surfaces of any genus.

We also use these moduli spaces to distinguish Legendrian surfaces with the same classical invariants.
Generally, much less is known about Legendrian surfaces than Legendrian knots.  The first examples of inequivalent Legendrian knots with the same classical invariants were obtained by Chekanov and distinguished by the Chekanov-Eliashberg differential graded algebra (dga), which can be computed
combinatorially.
In the case of Legendrian surfaces, the dga
is much more difficult to compute.  The best technology for enumerating the necessary holomorphic disks
is Ekholm's gradient flow trees \cite{E}.
Very recent work of Rutherford and Sullivan \cite{RS} exploits this technology to
reduce the computation of the dga of many Legendrian surfaces to a (still difficult) combinatorial procedure.  Our results, where the symplectic analysis is subsumed by the local, combinatorial nature of sheaves, demonstrate the strength of microlocal sheaf techniques in symplectic topology.

In addition, we are able to use the Lagrangian structure of this moduli space inside a period domain
to perform an Aganagic-Vafa-style mirror symmetry \cite{AVms} and compute conjectural open Gromov-Witten invariants
of the obstructed non-exact Lagrangians which fill our Legendrian surfaces.

We now explain these results in some more detail.

\subsection{Legendrian surfaces, fillings, and constructible sheaves}
Our Legendrian surfaces lie in an open domain of $S^5$ contactomorphic to
the cosphere bundle $T^{\infty}\bR^3$ and to the jet bundle $J^1(S^2)$.  They are genus-$g$ surfaces double-covering their base projection to $S^2$ with $2g+2$ branch points.
Such a Legendrian $\leg$ can be defined from a cubic planar graph $\Gamma \subset S^2$, by constructing a wavefront projection of $\leg$
in $S^2\times \bR \cong \bR^3 \setminus \{0\}$ that is generically two-to-one over the base projection to $S^2$, but one-to-one over $\Gamma.$
(The original $\leg,$ not just its wavefront, is in fact still $2:1$ over the edges and only $1:1$ over the $2g+2$ vertices.)
\begin{figure}[H]
\includegraphics[scale = .25]{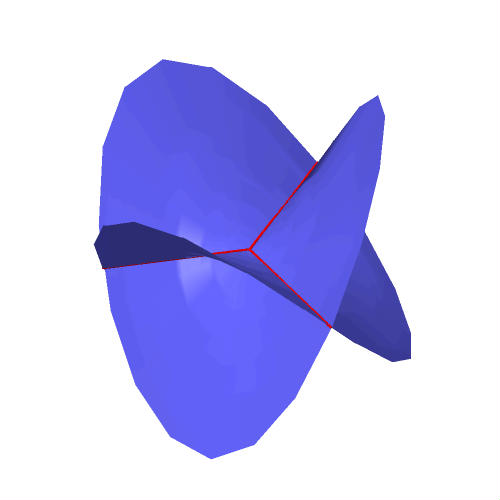}
\includegraphics[scale = .25]{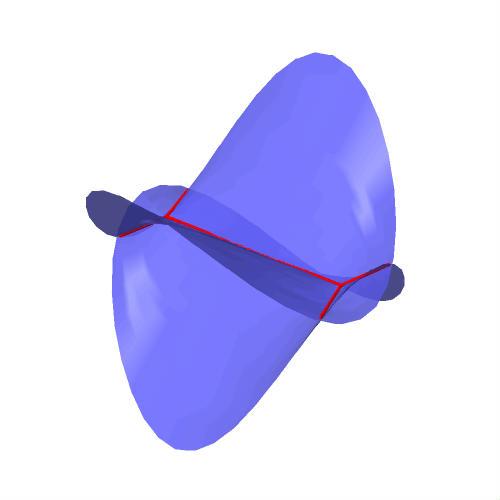}
\includegraphics[scale=0.3]{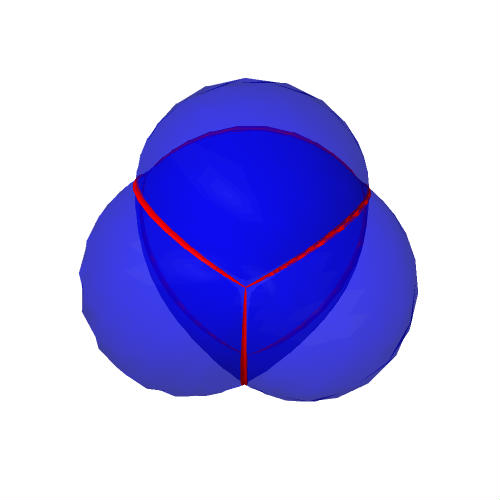}
\caption{Wavefronts near a vertex, edge, and at right for $\Gamma$ a tetrahedron.}
\label{fig:legprojs}
\end{figure}

\begin{remark}
The construction can be motivated as a higher-dimensional analogue of the Stokes Legendrian
\cite{STWZ} encoding the wild character varieties of complex curves near an irregular singularity of an holomorphic ODE.  For differential equations of second order, such a Legendrian can be specified by a finite set of points on the boundary, the locations where the asymptotic behavior of the two solutions switch.  In three dimensions, those points are replaced by cubic graphs (indeed the story in this paper has a generalization to arbitrary $3$-manifolds with cubic graphs drawn on their boundary) --- it would be very interesting to find a family of second-order PDEs whose asymptotic behavior exhibited the kind of Stokes phenomenon modeled on these graphs.
\end{remark}

 In this introduction, we will mainly restrict our attention to \emph{simple} (no loops or multiple edges) cubic planar graphs,
although the Legendrian surface associated to $\Gamma$ is of interest  more generally.

Given such a Legendrian $\leg \subset T^{\infty}\bR^3$ we can take it as singular-support data for a category of
constructible sheaves.   This category of constructible sheaves is equivalent \cite{NZ,N2} to a Fukaya category in $T^*\bR^3$ with asymptotic conditions on geometric branes defined by $\leg$.  We give an explicit description of the moduli space $\cM_r$ of objects of this category that have
microlocal rank $r$, focusing on $\cM := \cM_1.$  In \cite{KS}, sheaves in $\cM_r$ are called ``pure,'' or ``simple'' if $r = 1$.  
In the Fukaya category, points of $\cM_r$ correspond to stacks of $r$ basic branes ending on $\leg$ --- see Remark \ref{rem:boco}.

A point of $\cM$ is easy to describe:  think of the faces of $\Gamma$ as countries on a map of the globe $S^2,$ and color each one with a
point in $\bP^1$ so that countries which share an edge border have different colors.  The set of such choices provides a framed version
of the moduli space:  to get $\cM$, we must quotient by the automorphism group
$\PGL_2$ of $\bP^1$, which acts freely.  If we put $\widehat \Gamma$ for the dual graph,
this means $\cM$ is the set of graph colorings of $\widehat \Gamma,$ with colors chosen in $\bP^1,$ modulo $\PGL_2$.  As a variety, $\cM$ is defined over the integers and we may count its points over finite fields:

\begin{proposition}
\label{intro-prop-1}
Let $\Gamma$ be a simple, cubic planar graph, $\widehat \Gamma$ its dual graph. Let $P_{\widehat{\Gamma}}$ denote the chromatic polynomial, whose value $P_{\widehat{\Gamma}}(c)$ is the number of colorings of $\widehat{\Gamma}$ with $c$ colors.
Let $\bF_q$ be a field with $q$ elements.  Then 
\[
\#\cM/\bF_q = \frac{1}{q^3-q}\cdot P_{\widehat{\Gamma}}(q+1).
\]
\end{proposition}
\begin{proof}
The argument $q+1$ is the number of $\bF_q$-points of $\bP^1$.  The denominator is the order of $\PGL_2(\bF_q)$.  
\end{proof}

It is known \cite{Birkhoff} that if $G$ has $n$ vertices and $e$ edges, then $P_G(x) = x^n - ex^{n-1} + O(x^{n-2}).$
If $\Gamma$ is a cubic planar graph with $v$ vertices, $e$ edges and $f$ faces, then
it is easy to show that $$v = 2g + 2,\quad e = 3g + 3,\quad f = g + 3,$$
where $g$ is the genus of the corresponding Legendrian surface $\leg.$

\begin{theorem}
\label{intro-thm-2}
Let $\leg\subset T^{\infty}\bR^3$ be the genus-$g$ Legendrian surface defined by a
simple, cubic planar graph $\Gamma$.  Then $\leg$ has no
smooth oriented graded exact Lagrangian fillings in $\bR^6.$
\end{theorem}

\begin{proof}
An oriented 3-manifold $L$ whose boundary is a genus-$g$ surface has $b_1(L) \geq g$.  One automatically has $w_2(L) = 0$, so that if $L$ is a graded exact Lagrangian filling, the inverse of microlocalization \cite{NZ} defines a torus chart $(\Gm)^{\times g} \hookrightarrow \cM$ as in \cite{STWZ}.
Thus over $\bF_q$, a filling gives $(q-1)^g = q^g - gq^{g-1} + \cdots$ distinct points of $\cM$.
On the other hand, recalling that for $\widehat{\Gamma}$ we have $n = g+3$ and $e = 3g+3$, we get
$$\#\cM/\bF_q = \frac{1}{q^3-q}\left( (q+1)^{g+3} - (3g+3)(q+1)^{g+2} + \cdots \right) = q^{g} - 2g\cdot q^{g-1} + \cdots$$
Taking $q$ large, we see $\cM$ is not big enough to accommodate any torus chart!\end{proof}

Nevertheless, we can construct smooth fillings that are not exact.  Suppose that $\Gamma$ is drawn on the surface of the ball $D^3$.  A \emph{foam} filling $\Gamma$ is a singular surface $\bF \subset D^3$ with codimension-one singularities that look like the letter ``Y'' times an interval, and codimension-two singularities (points) that look like the cone over the $1$-skeleton of a tetrahedron, and such that $\bF \cap \partial D^3 = \Gamma.$

\begin{construction}
From a foam in $D^3$ we construct an exact \emph{singular} Lagrangian filling of $\leg$.  The singularities are of Harvey-Lawson type, and can be smoothed away (in different ``phases'') to smooth non-exact Lagrangian fillings.  
\end{construction}

A foam is the same combinatorial structure that configurations of soap bubbles have --- the singularities are the Plateau borders of the soap film.  It is interesting to speculate whether, if we were to impose Plateau's laws on the foam (that the sheets of soap have constant mean curvature and meet at equilateral angles), the Lagrangian we construct could be chosen to be sLags.
The singular Lagrangians are easy to describe, they are branched double covers of the $1$-skeleton of the foam.  The smooth Lagrangians we construct are not just non-exact but \emph{obstructed}: they are not objects in a Fukaya category if one does not introduce a ``bounding cochain.''

\begin{remark}
\label{rem:boco}
In the Fukaya category constructed in \cite{NZ}, only exact branes were considered, but in most of our examples such exact Lagrangians are not available to support points of $\cM$ or $\cM_r$.  Instead, points of $\cM_r$ morally correspond to geometric Lagrangians, usually non-exact and even obstructed, together with a
$\mathrm{U}(r)$-connection.  In the two-dimensional version of this analogy exploited in e.g.~\cite{STWZ}, the connection can be taken to be flat, but in
three dimensions when the Lagrangian is obstructed it must obey a more complicated Maurer-Cartan equation. 
This is part of the theory of bounding cochains and curved $A_\infty$-categories, which so far has not influenced microlocal sheaf theory.
\end{remark}


In any case, there is a map from $\cM$ that corresponds to ``restrict a local system to its boundary'' when an exact Lagrangian
does exists, and makes sense over the whole moduli space.  It is a special case of the $\mu\mathit{hom}$-functor of \cite{KS}, that
was called ``microlocal monodromy'' in \cite{STZ}.  This map 
\begin{equation}
\label{eq:lagemb}
\phi:\cM \to H^1(\leg,\bC^*)
\end{equation}
is isomorphic to one considered in \cite{DGGo}, and is known to be Lagrangian with respect to the symplectic structure on $H^1(\leg,\bC^*)$ induced by the intersection form --- something that follows in this sheaf-theoretic context from the general results of \cite{BD,ST}.

\subsection{Generalized Aganagic-Vafa mirror symmetry}
Recall that Aganagic-Vafa \cite{AV} fix a Lagrangian brane, then use the equation of the moduli space of that brane in the resolved conifold
to define conjectural open Gromov-Witten invariants.  In \cite{AKV} the construction was extended
to different ``phases'' and ``framings'' of the brane, and to other toric Calabi-Yau three-folds, including $\bC^3.$
In \cite{AVflop} the authors apply this construction to
conormals of knots and call this ``generalized SYZ mirror symmetry.'' 
Our method generalizes this technique
to three-dimensional branes that look nothing like tori, so we refer to it as ``generalized AV mirror symmetry.''

The original AV construction identifies a distinguished set of coordinates $x = e^u, y = e^v$
on an ambient symplectic $(\bC^*)^2$
and uses the equation of the moduli space $F(u,v) = 0$ to solve $v = -\partial_u W$ (the choice of sign is historical), where $W$ has a
four-dimensional interpretation as a superpotential, but also as the generating function for open disk invariants.
The procedure depends on a choice of ``framing'' which distinguishes the second coordinate, as the so-called ``phase'' only identifies
one of the two.

The method can be interpreted in the following way.  The moduli space $\cM$ of the brane is a Lagrangian submanifold
in $H^1(\leg,\bC^*) = (\bC^*)^{2g}$.  A \emph{phase} is a geometric Lagrangian, an oriented $3$-manifold $L$, contained in $\bR^6$ with boundary on $\leg$.  A phase determines a Lagrangian map $H^1(L,\bC^*) \to H^1(\leg,\bC^*)$, and we define a framing to be an extension of this to a symplectic map $T^* H^1(L,\bC^*) \to H^1(\leg,\bC^*)$.  Then it makes sense to try to express $\cM$ as the graph of the differential of a multivalued transcendental function $W:H^1(L,\bC^*) \to \bC$.
\begin{quote}
\emph{We conjecture that $W$ is the generating function of open Gromov-Witten disk invariants.}
\end{quote}
Currently, the Lagrangians considered in this paper fall outside the limited class for which open Gromov-Witten invariants are rigorously defined, so the conjecture is not strictly precise as it stands.  However, there is recent progress in the work of Solomon and Tukachinsky --- see Remark \ref{rem:jake-sara}.  Meanwhile, we can make a precise integrality conjecture, following Ooguri-Vafa \cite{OV}:
\begin{quote}
There are integers $a(d,f) \in \bZ$ and an order-two element $\varepsilon_f \in H^1(L,\bC^*)$, indexed by homology classes $d \in H_1(L,\bZ)$ and a framing $f$, such that
\begin{equation}
\label{eq:xmasW}
W(x) = \sum_d a(d,f) \Li_2\left((\varepsilon_f x)^d\right)
\end{equation}
where $(-)^d$ denotes the monomial function $H^1(L,\bC^*) \to \bC$ corresponding to $d$ and $\Li_2(z) = \sum \frac{1}{n^2} z^n$ is the classical dilogarithm function.
\end{quote}
Identities involving $\Li_2(z)$, $\Li_2(1/z)$, and other arguments related by M\"obius transformations prevent equation \eqref{eq:xmasW} from determining $a(d,f)$ uniqely.  We expect that by restricting the sum to a strictly convex (``Mori'') cone of elements in $H_1(L,\bZ)$ that support holomorphic disks, the coefficients $a(d,f)$ become uniquely determined and are the physical BPS numbers.  The translation by $\varepsilon_f$ functions as a kind of mirror map (change of variables to write superpotential correctly), that we do not have a great explanation for.

Before giving some examples let us discuss coordinates.  The choice of phase allows us to choose coordinates $(v_1,\ldots,v_g)$ on the universal cover of $H^1(\leg,\bC^*)$ that cut out $H^1(L,\bC)$, and a framing nails down conjugate coordinates
$u = (u_1,...,u_g)$.  These
choices identify the universal cover with $\bC^{2g} \cong T^*\bC^g$, and $W$ is the solution to $v_i = \pm \partial_{u_i} W$.

%
%
%

\medskip
We implement this generalized AV mirror symmetry explicitly in several examples.

\subsection{Examples}

\begin{itemize}
\setlength{\itemsep}{10pt}
\item When $\Gamma$ is the complete graph with four vertices, we have $f = g + 3 = 4,$ so $g=1$ and the filling is a solid torus.
It is the Aganagic-Vafa brane in $\bC^3$.  Then $\cM \subset (\bC^*)^2$ is defined by $$x + y = 1,$$ a pair of pants,
and here
we have implicitly chosen a phase and (zero) framing to write $x = e^u, y = e^v.$  Then $v = -\partial_u W$ identifies
$W = {\rm Li}_2(x) := \sum_{n>0}\frac{1}{n^2}x^n$, as in the work of Aganagic-Vafa \cite{AV}.  The procedure is identical to theirs in this example,
so different framings --- making the replacement $u \to u + pv$ --- give the framing-dependent
disk invariants computed by Aganagic-Klemm-Vafa \cite[Section 6.1]{AKV}.
\item When $\Gamma$ is a triangular prism, $g = 2$ and the moduli space $\cM$ is cut out by two equations:
$$x_1 + y_1 = 1,\qquad x_2 + y_2 = 1.$$
Then $\cM$ is a product of two copies of the pair of pants, and zero framing ($x_i = e^{u_i}, y_i = e^{v_i}$) gives
$W = {\rm Li}_2(x_1) + {\rm Li}_2(x_2).$
``Diagonal'' framings $u_1 \to u_1 + p_1 v_1, \; u_2 \to u_2 + p_2 v_2$ lead to different BPS numbers, but exactly as above.
Here, however, we also have the possibility of non-diagonal framings such as
$$u_1 \to u_1 + p v_2, \qquad u_2 \to u_2 + p v_1.$$
For example, when $p = 1$ we find $W(x_1,x_2) = {\rm Li}_2(x_1) + {\rm Li}_2(x_2) - {\rm Li}_2(x_1x_2).$
More generally, framings are determined by a symmetric
$2\times 2$ integer matrix $M_{i}{}^j$ via $u_i \to u_i + M_{i}{}^j v_j.$
We consider these examples in Section \ref{sec:tent}.
While we generally cannot write the superpotential $W$ in closed form, we can perform strict integrality checks:
in all cases considered, we derive integer BPS numbers $a(d,f)$
using the Ooguri-Vafa $1/d^2$ multiple cover formula \cite{OV}.
\item When $\Gamma$ is the union of vertices and edges of a cube, $g=3$ and we can similarly define
multi-coordinates $(u,v)$ for $\bC^6\cong T^*\bC^3_u$.  Then (the universal cover of) $\cM$ is a graph of $-dW(u),$ where
$$W = {\rm Li}_2(x_1) + {\rm Li}_2(x_2) + {\rm Li}_2(x_3) - {\rm Li}_2(x_1x_2) - {\rm Li}_2(x_1x_3),$$
with $x_i = e^{-u_i}.$  The answer is written in terms of integer linear combinations of dilogarithms, meaning the conjectural disk
invariants obey integrality.
\end{itemize}

\begin{remark}
Cubic planar graphs define triangulations of the plane, and graph mutations correspond to Pachner 2-2 moves.
This and the appearance of dilogarithms comes from an intimate relationship between the present work and cluster
theory.  In future work with Linhui Shen \cite{STZ2}, we will describe the process of mutation and the relationship of $W$ to the
DT series in cluster theory.
\end{remark}

\subsection{Relation to previous works in physics}
\label{sec:priorwork}

While preparing this document, we learned of prior constructions with signifcant overlap to our own.
We will try to explain the connections, similarities and differences to the present work.
While the works in Section \ref{sec:dimofteetal} appeared first, we describe the relations in reverse chronological
order due to the greater similarity in the approach of the latter works.

\subsubsection{The work of Cecotti, C\'ordova, Espahbodi, Haghighat, Neitzke, Rastogi and Vafa}
\label{sec:vafaetal}
 
 In the series of works \cite{CCV,CNV,CEHRV}, the above-named authors consider the dimensional reduction
 of the theory on a stack of two overlapping M5-branes wrapping a three-dimensional spacetime cross a Lagrangian three-fold.
 The two recombine into a single-M5 brane, the Lagrangian being a branched double cover over $\bR^3.$
 The authors also consider the Harvey-Lawson cone and singular and smooth tangles --- see in particular
 Section 5.1.2 of \cite{CCV} and Section 3.1 of \cite{CEHRV}.
 
 Those authors consider Seifert surfaces for the tangles, which they use to construct actions for the effective
 three-dimensional theory.\footnote{We use the word ``action'' to avoid overloading the word ``Lagrangian.''}
 This is an interacting gauge theory, which in the Coulomb branch has a $\mathrm{U}(1)$ gauge field for
 each homology loop of the three-manifold, with Chern-Simons levels determined by the
 self-linking matrix of the tangle.  Chiral matter and superpotentials have M-theoretic descriptions
 via holomorphic maps with various boundary conditions.
 The guiding principle is that the effective physics should be independent of the Seifert surface
 used to describe the three-manifold and construct the action:  different Seifert surfaces thus determine dual theories.
 Equivalence of these theories is shown to give a physical explanation of the wall-crossing formula of Kontsevich-Soibelman.
 
 In the three-dimensional theory, BPS states are formed by M2-branes ending on the M5-brane,
 so holomorphic disks bounding circles in the Lagrangian.  The circles surround two strands of the tangle,
so massless states arise for singular tangles, when the disk shrinks to zero size.  The fields of the three-dimensional
theory that create such states are chiral multiplets.

M2-brane instantons give rise to superpotential terms which couple the massless particles.  The boundary of
the three-ball M2 brane is a two-sphere which projects to a polygon, the sides of which label the different
coupled chiral multiplets.

\subsubsection{The work of Dimofte, Gabella, Gaiotto, Goncharov, Gukov and Hollands}
\label{sec:dimofteetal}

As just mentioned in Section \ref{sec:priorwork}, the above-named authors also consider M5-branes defined
by a three-dimensional Lagrangian \cite{DGGo, DGGu, DGH}, usually a hyperbolic manifold
such as a branched double cover of a knot.
(These authors also consider $n$-fold covers, but not through a description of its branch locus.)
They define an isomorphic moduli space $\cM$ and map \eqref{eq:lagemb}.
They investigate the relationship between, on the one hand, Chern-Simons theory and hyperbolic three-manifolds, and on the
other, three-dimensional supersymmetric field theories, their supersymmetric indices, superpotentials, and relation to four-dimensional
BPS states.  Some sketches are given in Appendix \ref{app:physics}.  

The authors construct a moduli space of $G$-local systems on a three-manifold $L$ with an ideal triangulation --- i.e.
a decomposition into hyperbolic tetrahedra with boundary surfaces --- where $G = \PGL_n(\bC).$
Gluing these together, one arrives at either a closed three-manifold or
one with a boundary surface.  For example, a boundary torus is relevant to the knot complement of a hyperbolic knot, and this
is a prime example.  The tetrahedra are truncated and thus have four triangles near the vertices, and one assigns an element of $\bP^1$
at each triangle.  The space $\mathrm{Loc}_G(\partial L)$ is known to be symplectic, and the restriction of a local system to the boundary
embeds $\mathrm{Loc}_G(L)$ as a Lagrangian submanifold.

\medskip

In forthcoming work with Shen \cite{STZ2}, we use cluster theory to determine the wavefunctions/partition functions
appearing in these physical settings, including the dependence on framings.
The present work gives a conjectural relationship between these functions and the framing-dependent superpotentials 
encoding open Gromov-Witten invariants.

%
%

\vspace{2mm}\noindent {\bf Acknowledgements.}
It is a pleasure to thank Roger Casals, Bohan Fang, Xin Jin, Melissa Liu, Emmy Murphy, David Nadler, and Linhui Shen for sharing their insights, time and suggestions. 
Roger Casals and Emmy Murphy were quickly able to settle many questions raised in an early draft of this paper.
We also thank Tudor Dimofte and Cumrun Vafa for helpful discussions about their respective joint works. 
We thank Jake Solomon and Sara Tukachinsky for discussing the role of framings in their work.
D.T. is supported by NSF grant DMS-1510444.   E.Z. is supported by NSF grant DMS-1406024.

\section{The hyperelliptic wavefront}
\label{sec:wavefront}

As in the introduction, let $\Gamma \subset S^2 \subset \bR^3$ be a cubic planar graph.  We assume each edge is smoothly embedded, and that in the tangent space to a vertex the three edges are linearly independent and do not lie in a half-plane.  We write $v,e,f$ for the number of vertices, edges, and faces of $\Gamma$.  Euler's relation gives $v -e  +f = 2$ and the cubicness of $\Gamma$ gives $3v = 2e$, so there is an integer $g$ such that  
\begin{equation}
\label{eq:fveg}
f = g+3 \qquad v = 2g+2 \qquad e = 3g+3
\end{equation}

Let $\pi:\leg \to S^2$ be the connected oriented double cover of $S^2$, branched over the vertices of $\Gamma$ --- the Riemann-Hurwitz formula shows $\leg$ has genus $g$.  Write $\iota_\leg$ for the nontrivial Deck transformation of $\pi$.

\begin{definition}
\label{def:hyperelliptic-wavefront}
We define a \emph{hyperelliptic wavefront} modeled on $\Gamma$ to be a map $i:\leg \to \bR^3$ with the following properties:
\begin{enumerate}
\item The image of $\leg$ does not contain the origin so that $i$ can be written in spherical coordinates $\leg \to S^2 \times \bR_{>0}$.
\item In those coordinates we have $i(x) = (\pi(x),r(x))$ where $\pi$ is the covering map and $r:\leg \to \bR_{>0}$.
\item The map $r$ obeys $r(\iota_\leg(x)) = 1/r(x)$, with $r(x) = 1$ exactly on the preimage of $\Gamma$.
\item Near each vertex of $\Gamma$, we may find coordinates $(x,y)$ on $S^2$ and $(u,v)$ on $\leg$ such that $(\pi,r)$ is given by
\begin{equation}
\label{eq:vertexequation}
x(u,v) = u^2 - v^2 \qquad y(u,v) = 2uv \qquad r(u,v) = \exp(2u^2 v - \frac{2}{3} v^3)
\end{equation}
In other words, after setting $w = u + iv,$ we have $x+iy = w^2$ and $\log(r) = \frac{2}{3} \mathrm{Im}(w^3)$.
\item $r$ has exactly $(2g+2)+2(g+3)$ critical points: the $2g+2$ critical points of $\pi$, and a unique maximum and unique minimum over each face of $\Gamma$.
\end{enumerate}
\end{definition}

There is a hyperelliptic wavefront modeled on every $\Gamma$.  One way to construct it is to identify $S^2$ with the Riemann sphere (with holomorphic coordinate $z$, say), and choose a Strebel differential $f(z)dz^2$ whose non-closed trajectories are the edges of $\Gamma$.  Such a differential will have a unique quadratic pole $p_i$ in each face, and the Strebel condition implies that $\tilde{r}(z) = \exp(\pm \mathrm{Im}\int\sqrt{f(z)} dz)$ obeys conditions (1)-(5) --- except it takes the values $\{0,\infty\}$ on the poles $p_i$.  We obtain $r$ by damping $\tilde{r}$ near each pole $p_i$.  We call this the ``Strebel model'' for the hyperelliptic wavefront.

The local model at a vertex \eqref{eq:vertexequation} gives an immersed hypersurface with double points along $\mathrm{Im}(w^3) = 0$, shown here in red.
\begin{center}
\includegraphics[scale = .25]{localmodel.jpeg}
\end{center}
It fails to be an immersion at $w = 0$, but it has a well-defined normal direction there: radially outwards.  The oriented Legendrian lift (also called the oriented Nash blowup) gives a smooth Legendrian embedding to $T^{\infty} \bR^3 \cong \bR^3 \times S^2$.  The Nash blowup construction shows that one may recover a hyperelliptic wavefront $i:S \to \bR^3$ from the image $\Psi := i(S)$ as a subset of $\bR^3$.

\begin{remark}
\label{rem:polar-bear}
A few comments are in order.
\begin{enumerate}
\item 
The wavefront immersion is not generic among front projections --- it is possible to ``push three swallowtails'' out of any vertex, either on the top or the bottom of the figure.  In other words it is the critical front of a ``bifurcation of fronts.''  In low dimensions Arnold gave a classification of such bifurcations, and this one is called $\mathrm{D}_4^{-}$.
\item When $\Gamma$ is not cubic, but the tangential angles between adjacent edges at a degree $d$ vertex are $(360/d)^{\circ}$, it is still possible to construct $\leg$ and a wavefront immersion to $\bR^3$, but it is no longer smooth over the vertices.  The local model (4) at a degree $d$ vertex looks like the Legendrian lift of the plane curve singularity $w^2 = z^{d-2}$. 
\item \label{third} We define the ``Lagrangian projection'' of a hyperelliptic wavefront to be the projection to $T^* S^2$ whose cotangent coordinates are given by the gradient of $\log(r)$.  The Lagrangian projection is a Lagrangian immersion with $g+3$ double points.  In the Strebel model, there is a neighborhood $U_{\Gamma} \subset S^2$ of $\Gamma$ for which the Lagrangian projection $\pi^{-1}(U_{\Gamma}) \to T^* S^2 = T^* \bP^1$ is a holomorphic embedding, but the immersion $S \to T^* S^2$ cannot be made holomorphic.

\item \label{fourth} Suppose $M$ is a $3$-manifold, $\Sigma \subset M$ is a two-sided smooth surface in $M$, and $\Gamma \subset \Sigma$ is a cubic graph.  If $\leg$ denotes the double cover of $\Sigma$ branched over the vertices of $\Gamma$, we have an associated wavefront $i:\leg \to M$, with $i(S) \cap \Sigma = \Gamma$.  Of course, the name ``hyperelliptic'' is less appropriate when $\Sigma$ is not a sphere.  An interesting general class of examples are suggested by the Legendrian knots in the cocircle bundles over surfaces, considered in \cite{STWZ}.  If $\partial M \times [0,1] \to M$ is a collar neighborhood of the boundary, then we can set $\Sigma = \partial M \times \{\frac{1}{2}\}$ --- we refer to these as ``Stokes wavefronts of degree $2$,'' although we do not know how to define Stokes wavefronts of higher degree.  The wavefronts of Definition \ref{def:hyperelliptic-wavefront} are included in the class of Stokes wavefronts, they are just the case when $M$ is a ball.

\end{enumerate}
\end{remark}

When $\Gamma$ is the edge graph of a tetrahedron, the hyperelliptic wavefront is as shown in Figure \ref{fig:cloudytet} below.
\begin{figure}[H]
\includegraphics[scale=0.3]{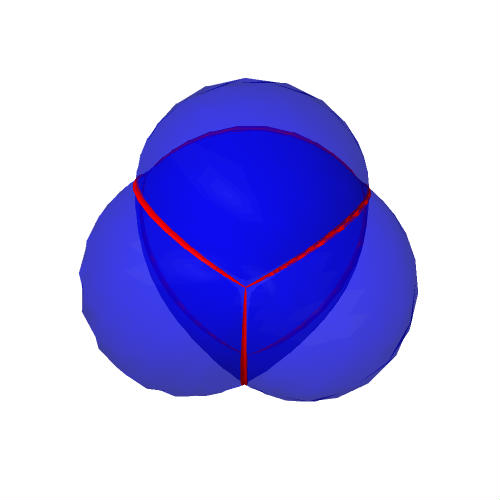}
\caption{The hyperelliptic wavefront defined by the edge graph (in red) of a tetrahedron.}
\label{fig:cloudytet}
\end{figure}

\subsection{The extended Legendrian $\leg^+$}
\label{subsec:extended-leg}
In \cite[\S 2.2.1]{STZ}, we associated to each front diagram $\Phi \subset \bR^2$ a Legendrian graph $\Lambda^+$, obtained from the Legendrian lift $\Lambda$ of $\Phi$ by attaching a Legendrian chord joining the two points projecting to a crossing.  The cone over $\Lambda^+$ is the smallest Lagrangian containing the cone over $\Lambda$ that is closed under addition.  (Thus $\Lambda^+$ is ``cragged'' in the language of \cite{Nicolo}.)

We may consider a similar extended form of $\leg$, which we denote by $\leg^+$.  If $i$ is an edge of $\Gamma$, its preimage is parametrized by a circle $\gamma_i:S^1 \to \leg$, with a distinguished orientation we discuss in \S\ref{subsec:period-domain}.  If the edge $i$ is not a loop,
then $\gamma_i$ is an embedded circle, while if $i$ is a loop from vertex $v$ to itself, $\gamma_i$ is an immersion with a single double point at $\pi^{-1}(v)$.  We attach a Legendrian disk to $\leg$ whose boundary is along $\gamma_i$.  

To define $D_i$, first note that each point $p$ along $i$ has two associated outward conormal rays, one along each branch of the wavefront immersion.  $D_i$ is the subset of the cosphere bundle over $i$ which, at a point $p \in i$, consists of the acute angle of the conormal rays between those two.  Note this angle shrinks to zero at the endpoints of $i$, where $D_i$ is just a point.  Thus $D_i$ is the disk formed by the family of intervals over the interval $i,$ collapsing at the boundary.

\subsection{Edge and disk moves}
\label{subsec:edge-and-disk-moves}

If $i$ is an edge of $\Gamma$ that is not a loop, we may construct a new graph $\Gamma'$ by changing $\Gamma$ in a neighborhood of $i$ according to the following diagram
\begin{center}
\begin{tikzpicture}
\draw (-2,0)--(-1.5,.5);
\draw (-2.75,-.25)--(-2,0)--(0,-1)--(.75,-.75);
\draw (0,-1)--(-.5,-1.5);

\draw (1+4,0)--(1+3.5,.5);
\draw (1+4.75,-.25)--(1+4,0)--(1+2,-1)--(1+2-.75,-.75);
\draw (1+2,-1)--(1+2.5,-1.5);

\node at (-.5,0) {$a$};
\node at (-2.5,.5) {$b$};
\node at (.5,-1.5) {$d$};
\node at (-1.5,-.75) {$c$};

\node at (5.5,.5) {$a$};
\node at (3.5,0) {$b$};
\node at (4.75-.1,-1+.1) {$d$};
\node at (2.5,-1.5) {$c$};

\end{tikzpicture}
\end{center}
The faces of $\Gamma$ and of $\Gamma'$ are in natural bijection with each other.  By performing this move at different edges one can get from any graph to any other graph.  Indeed this is the sense in which cubic planar graphs label the top-dimensional cells in the Harer-Penner-Mumford-Thurston complex for $M_{0,n}$, which is connected.

Let $U_i \subset S^2$ be a neighborhood of $i$.  If $i$ is not a loop, the preimage of $U_i$ in $\leg$ is an annulus, and its preimage in $\leg^+$ is an annulus with a disk attached --- let us denote them by $A$ and $A^+$.   The Lagrangian projection of $A^+$ is an embedding into $T^* U_i$ --- it is an exact, \emph{singular} Lagrangian of the kind that has been considered by \cite{Yau}.  Yau considers a ``disk move'' which is part of our model for cluster transformations in \cite{STWZ,STW}.  In the present context these are related:

\begin{center}
\includegraphics[scale = .25]{edgeduck.jpeg} \reflectbox{\includegraphics[scale = .25]{edgeduck.jpeg}}
\end{center}

The two exact Lagrangian surfaces on either side of a disk move are part of the ``focus-focus'' family of Lagrangian annuli:

\begin{proposition}
\label{prop:same-classical-invariants}
Let $\leg$ and $\leg'$ be two Legendrian surfaces associated to $\Gamma$ and $\Gamma'$, with the same number of vertices.  Let $j$ and $j'$ denote the exact Lagrangian immersions into $T^* S^2$, each with $g+3$ double points.  Then $j$ and $j'$ are isotopic to each other through (non-exact) Lagrangian immersions that do not increase or decrease the number of double points.
\end{proposition}

In particular the Legendrians $S$ must have the same classical invariants.

\subsection{The Ekholm-Honda-K\'alm\'an Lagrangians}
\label{subsec:EHK}
Let $n$ be an integer, which for simplicity we will assume is odd.  Let $\Lambda \subset S^3$ be a Legendrian $(2,n)$-torus knot.  Two families of $\frac{1}{n+1} \binom{2n}{n}$ exact Lagrangian fillings of $\Lambda$ have been produced, first in \cite{EHK} and later in \cite{STWZ}.  Roughly speaking, \cite{EHK} produces this family of fillings by a ``elementary Lagrangian cobordism'' technique, and \cite{STWZ} by an ``alternating strand diagram'' technique.  We sketch here a third construction of a family of $\frac{1}{n+1}\binom{2n}{n}$ fillings, in terms of hyperelliptic wavefronts, and an equivalence between these fillings and the EHK fillings.

\subsubsection{Rainbow, spirograph and Ng projection}
As in \cite{STWZ} it is convenient to draw the $(2,n)$-torus knot in its ``spirograph projection,'' with $n+2$ crossings around a circle --- let us recall the correspondence here.  Let $\Lambda_n$ be the Legendrian $(2,n)$ torus knot.  There are two contactomorphisms
\[
{\qquad T^{\infty,-} \bR^2} \hookrightarrow S^3  \qquad \qquad \qquad \qquad T^{\infty} \bR^2 \hookrightarrow S^3,
\]
onto open subsets of $S^3$, and $\Lambda_n$ is contained in the image of both.  The front projection to $\bR^2$ has $n$ crossings and $4$ cusps in the first projection, and $n+2$ crossing but no cusps in the second projection.  For example when $n = 3$:
\begin{center}
\includegraphics[scale = .25]{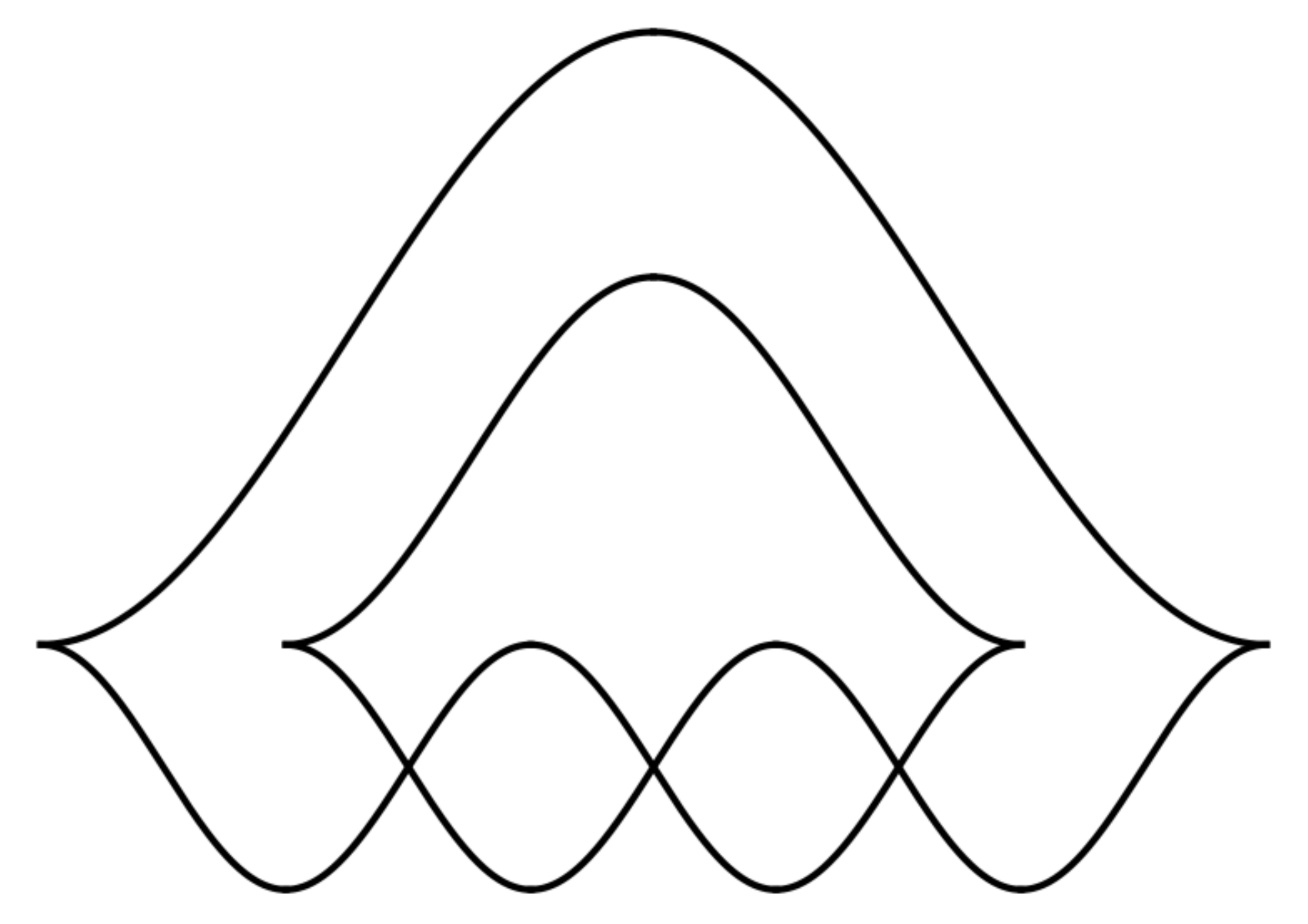}
\qquad \quad
\includegraphics[scale = .275]{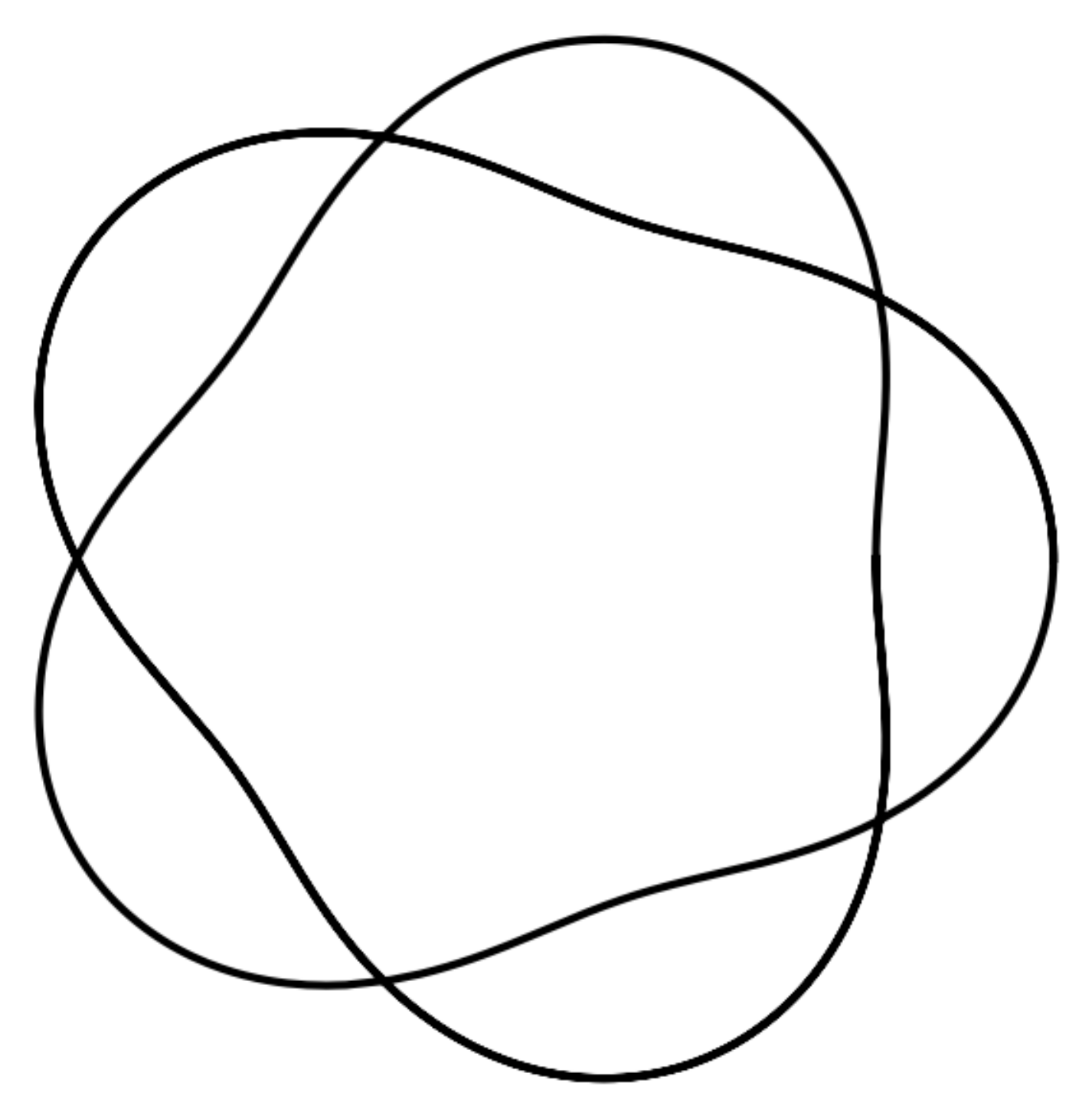}
\end{center}
The rainbow projection is Reidemeister-equivalent to the following reflection-invariant front, which we will call the \emph{Ng projection}:
\begin{center}
\includegraphics[scale = .33,angle = 180]{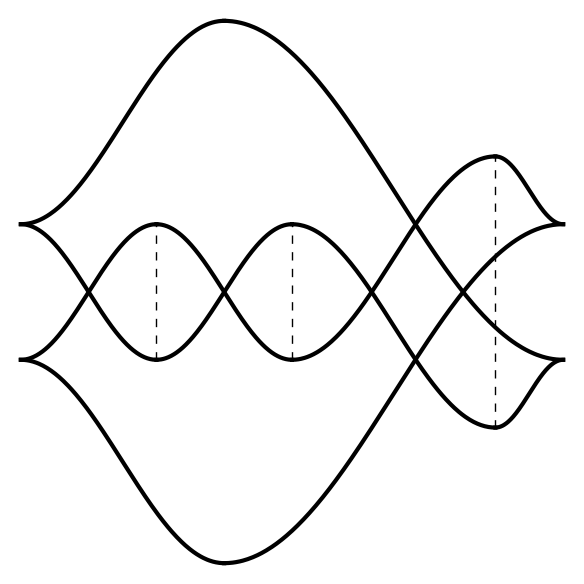}
\end{center}
This front is homeomorphic to one whose Lagrangian projection is \cite[Fig. 2, p. 5]{EHK} The dashed lines indicate the Reeb chords that loc. cit. uses to construct saddle cobordisms.

\subsubsection{Construction of fillings}

We replace $S^2$ with a disk $D^2$, and $\Gamma$ with a trivalent tree with $n$ vertices, whose $n+2$ leaves reach the boundary of the disk at those $n+2$ crossings --- we further assume that in a collar neighborhood of $\partial D^2$, each leaf is a radial interval.
A branched double cover $\leg$ of $D^2$ over the vertices of the tree is a genus $\frac{1}{2}(n-1)$ surface with one boundary component.  To give a hyperelliptic wavefront of $S$ into $D^2 \times \bR$, we give a function $f:S \to \bR$ analogous to the logarithm function $r$ of Definition \ref{def:hyperelliptic-wavefront}.  We replace condition \ref{def:hyperelliptic-wavefront}(5) by the condition that $f$ has no critical points besides the vertices of $\Gamma$, and that in a collar neighborhood of $D^2$ it is linear in the radial coordinate of $D^2$.  These conditions ensure that the Lagrangian projection --- i.e. the map $S \to T^* D^2$ given by $(\pi,\nabla f)$, as in Remark \ref{rem:polar-bear}\eqref{third} --- is an embedding, with Legendrian boundary at $T^* D^2\vert_{\partial D^2}$.

\begin{center}
\begin{figure}[H]
\includegraphics[scale = .25, angle = 270]{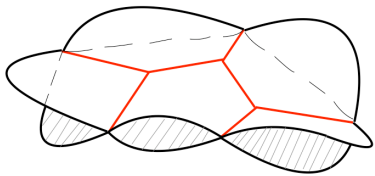}
\caption{A tree graph $\Gamma$ in red, and wavefront for an exact Lagrangian filling of a Legendrian trefoil knot.}
\label{fig:Gammafill}
\end{figure}
\end{center}

\subsubsection{Comparison to the Ekholm-Honda-K\'alm\'an construction}
\label{subsubsec:223}
Ekholm-Honda-K\'alm\'an observed that a Reeb chord in the Ng projection of the $(2,n)$ torus knot determines a ``saddle cobordism'' \cite[\S 6.5]{EHK} to the $(2,n-1)$ torus link, with one less Reeb chord in its Ng projection.  A total ordering of the Reeb chords gives a cobordism to the $(2,1)$ torus knot, which is a Legendrian unknot and has a unique disk filling.  They observed that many of these cobordisms were Hamiltonian isotopic to each other, making at most $\frac{1}{n+1}\binom{2n}{n}$ Hamiltonian isotopy classes.\footnote{Furthermore, they showed that the $\bF_2$ augmentation provided an invariant that showed they make at least $\frac{1}{3}(2^{n+1} - 1)$ Hamiltonian isotopy classes.  In \cite{STWZ} we used constructible sheaves to give an invariant distinguishing the Catalan-numbers-worth of fillings constructed there.}

A similar factorization is visible in the pictures of trees.  A planar tree is dual to an ideal triangulation of the disk with $n+2$ vertices.  Those vertices are in natural bijection with the Reeb chords of the spirograph projection.  Ordering the vertices determines a labeled, ideal triangulation of the disk with $n$ triangles, by slicing off triangles in order:
\begin{center}
\begin{tikzpicture}
\draw[ultra thick] (1.00, 0.000)--(0.309, 0.951)--(-0.809, 0.588)--(-0.809, -0.588)--(0.309, -0.951)--(1,0);
\node[right] at (1,0) {$1$};
\node[left] at (-0.809, 0.588) {$2$};
\node[left] at (-0.809, -0.588) {$3$};
\node[below] at (0.309, -0.951) {$4$};
\node[above] at (0.309, 0.951) {$5$};

\draw[ultra thick] (4+1.00, 0.000)--(4+0.309, 0.951)--(4-0.809, 0.588)--(4-0.809, -0.588)--(4+0.309, -0.951)--(4+1,0);
\draw[thick] (4+0.309, 0.951)--(4+0.309, -0.951);
\node[right] at (4+1,0) {$1$};
\node[left] at (4-0.809, 0.588) {$2$};
\node[left] at (4-0.809, -0.588) {$3$};
\node[below] at (4+0.309, -0.951) {$4$};
\node[above] at (4+0.309, 0.951) {$5$};

\draw[ultra thick] (8+1.00, 0.000)--(8+0.309, 0.951)--(8-0.809, 0.588)--(8-0.809, -0.588)--(8+0.309, -0.951)--(8+1,0);
\draw[thick] (8+0.309, 0.951)--(8+0.309, -0.951);
\draw[thick] (8+0.309, 0.951)--(8-0.809, -0.588);
\node[right] at (8+1,0) {$1$};
\node[left] at (8-0.809, 0.588) {$2$};
\node[left] at (8-0.809, -0.588) {$3$};
\node[below] at (8+0.309, -0.951) {$4$};
\node[above] at (8+0.309, 0.951) {$5$};

\node[red] at (8+.6,0) {$1$};
\node[red] at (8-.05,-.1) {$3$};
\node[red] at (7.5,.35) {$2$};

\end{tikzpicture}
\end{center}
The labeled triangulation does not depend on the ordering of vertices $n$, $n+1,$ $n+2.$
Write $a_k$ for the triangle labeled by $k$.
Then $a_n,\dots,a_1$ makes a shelling of the triangulated $(n+2)$-gon in the sense of \cite{BM}.  
In particular, the union $\bigcup_{k=1}^n a_k$ makes a triangulated $(n-k+2)$-gon, dual to a planar trivalent tree,
and the hyperelliptic wavefront associated to that tree gives a filling of a Legendrian $(2,n-k)$ torus knot or link.

\begin{remark}
\label{rmk:ehk}
The filling we construct only depends on the tree/triangulation, not on the labeling/shelling.
We gave an iterative description to highlight the connection to \cite{EHK}.
\end{remark}

\section{Foams and fillings}
\label{sec:foam}

In some sense we regard the hyperelliptic Legendrians of this paper as two-dimensional analogues of two-strand torus knots.  In this section we show that the format of \S\ref{subsec:EHK}, where planar trivalent trees give a blueprint for describing the Lagrangian surfaces that fill such knots, can be adapted to describe three-dimensional Lagrangians that fill hyperelliptic Legendrians.  The role of trivalent trees is played by foams:\footnote{Many related constructions
appear in the works described in Section \ref{sec:vafaetal} --- see, in particular, Section 5.1.2 of \cite{CCV} and Section 3.1 of \cite{CEHRV} ---
though not foams \emph{per se} or an emphasis on Legendrian boundaries.}

\begin{definition}
Let $D^3$ be the 3-dimensional ball.  A \emph{foam} in $D^3$ is a stratified subset $\bF^{2} \subset \bF^{1} \subset \bF \subset D^3$ where
\begin{enumerate}
\item As subsets of $D^3$, $\bF^{2}$ is a finite set of points, $\bF^{1} - \bF^2$ is a finite set of smoothly embedded arcs and $\bF - \bF^1$ is a finite set of smoothly embedded surfaces.
\item Near each point $\bF^2$, we may find smooth coordinates on $D^3$ identifying $\bF$ with the cone over the $1$-skeleton of a tetrahedron.
\item Each connected component of $\bF^1 - \bF^2$ has a neighborhood that looks like the cone on the vertices of a triangle, times an interval.
\end{enumerate}
We furthermore assume that $\bF$ is conical in a collar neighborhood $N$ of $\partial D^3$ --- it looks like the preimage of a cubic planar graph under the radial projection $N \to \partial D^3$.
\end{definition}

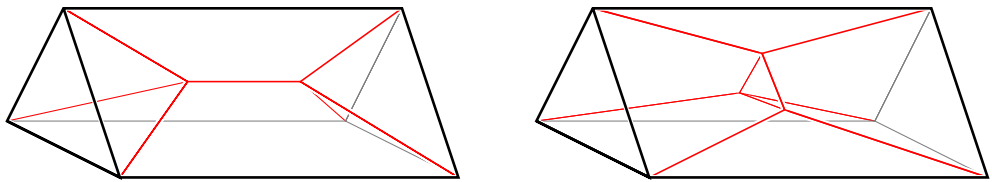
\begin{figure}[H]
\begin{center}
\begin{tikzpicture}[scale=1.5]
\tikzset{%
  channel/.style    = { white, double = black, line width = 1.2pt,
                     double distance = 0.4pt },
  link/.style = { white, double = black, line width = 0.4pt,
                     double distance = 0.6pt },}
\draw[channel,double=red] (1.6,.35) -- (0,0);
\draw[channel,double=red] (2.6,.35) -- (3,0);
\draw[gray] (0,0) -- (3,0) -- (3.5,1) -- (3,0) -- (4,-.5);
\draw[link,double=red, line width = .8pt] (1.6,.35) -- (.5,1) -- (1.6,.35) -- (1,-.5) -- (1.6,.35) -- (2.6,.35) -- (4,-.5) -- (2.6,.35) -- (3.5,1);
\draw[link,double distance = 1pt] (0,0) -- (1,-.5) -- (4,-.5) -- (3.5,1) -- (.5,1) -- (0,0) -- (1,-.5) -- (.5,1);
\draw[white, thick] (.9,-.534) -- (1.1,-.534);
\end{tikzpicture}
\qquad
\begin{tikzpicture}[scale=1.5]
\tikzset{%
  channel/.style    = { white, double = black, line width = .8pt,
                     double distance = 0.4pt },
  link/.style = { white, double = black, line width = 0.4pt,
                     double distance = 0.6pt },}
                     
\draw[channel,double=red] (0,0) -- (1.8,.25) -- (3,0);
\draw[channel,double=red] (2,.6) -- (1.8,.25) -- (2.2,.1);
\draw[red,line width = .3pt] (0,0) -- (1.8,.25) -- (3,0);
\draw[red,line width = .3pt] (2,.6) -- (1.8,.25) -- (2.2,.1);
%
\draw[gray] (0,0) -- (3,0) -- (3.5,1) -- (3,0) -- (4,-.5);
\draw[link,double=red, line width = .8pt] (.5,1) -- (2,.6) -- (2.2,.1) -- (4,-.5);
\draw[link,double=red, line width = .8pt] (1,-.5) -- (2.2,.1) -- (2,.6) -- (3.5,1);
\draw[red, line width = .6pt] (4,-.5) -- (2.2,.1) -- (2,.6) -- (.5,1);
\draw[link,double distance = 1pt] (0,0) -- (1,-.5) -- (4,-.5) -- (3.5,1) -- (.5,1) -- (0,0) -- (1,-.5) -- (.5,1);
\draw[white, thick] (.9,-.534) -- (1.1,-.534);
\end{tikzpicture}
\end{center}
\caption{Two examples of a foam filling a triangular prism graph.}
\label{fig:two-foams}
\end{figure}

The conditions of the definition are topological analogues of the Plateau conditions for soap films.  (The genuine Plateau conditions are metric --- soap films have constant mean curvature away from the singularities, where the smooth sheets meet at equilateral angles.)  A foam gives a regular cell complex structure on $D^3$, whose dual complex is a ``tetrahedronation'' of $D^3$ --- this is similar to the duality between planar trivalent graphs and triangulations.  We will say that a foam is \emph{ideal} if every connected component of $D^3 - \bF$ contains part of $\partial D^3$ --- then the dual tetrahedronization will have no internal vertices; this is analogous to the cubic planar graph being a cubic planar tree.

\subsection{The Harvey-Lawson cone}

\label{subsec:HL-cone}
The building-block example of a foam, is the cone over the $1$-skeleton of a regular tetrahedron centered at the origin.  Before turning to our general construction let us explain the relationship between this foam and the \emph{Harvey-Lawson cone}.  This is the Lagrangian subset $\mathit{HL} \subset \bC^3$ given parametrically by
\[
(r,e^{is},e^{it}) \mapsto (re^{is},re^{it},re^{-is-it})
\]
where $r\geq 0$ and $(s,t) \in S^1 \times S^1.$
It is a cone over a two-torus, with a conic singularity at the origin $r = 0$.  
Coordinatize $\bC^3$ as $z = x + iy,$ with $z = (z_1,z_2,z_3)$ and likewise for $x$ and $y$.
Then $\mathit{HL}$ is special Lagrangian with respect to the standard Calabi-Yau structure.
Define $n = \frac{x}{|x|} \in S^2$ and $f = n \cdot y \in \bR$
The restriction of ${\mathit HL}$ to fixed $r$ is Legendrian, with wavefront projection
$$(s,t) \mapsto e^f n \in \bR^3 \setminus \{0\} \cong S^2 \times \bR$$
shown in Figure \ref{fig:cloudytet}.  The Harvey-Lawson cone is therefore a singular Lagrangian
filling of the Legendrian surface associated to the tetrahedron.  (It is also exact and special.)

We will be interested here in a different projection, to the real three-space: $$\mathit{HL} \to \bR^3_x, \qquad (r,s,t) \mapsto (r\cos(s),r\cos(t),r\cos(-s-t)).$$
This is a double cover branched over the four rays 
\begin{equation}
\label{eq:those-rays}
\{(a,a,a) \mid a >0\}, \{(a,-a,-a) \mid a >0\}, \{(-a,-a,a) \mid a > 0\}, \{(-a,a,-a) \mid a > 0\}
\end{equation}
Note the generator for the $\bZ/2$ Galois group of this covering is given explicitly by $(r,s,t) \mapsto (r,-s,-t)$.  The unusually low dimension of the critical value locus --- here it has codimension $2$ --- is related to the fact that both $\mathit{HL}$ and the kernel of the projection $\bC^3 \to \bR^3$ are special Lagrangian of the same phase. 

The rays \eqref{eq:those-rays} are generated by the vertices of a regular tetrahedron.  The foam whose walls are the sectors through these edges play a role in the \emph{exactness} of $\mathit{HL}$.  
Indeed being a cone, the set $\mathit{HL}$ is contractible and therefore an exact Lagrangian with respect to any primitive $\alpha$ for $\omega$.  It is natural to take for $\alpha$ the canonical one-form obtained when identifying $\bC^3$ with $T^* \bR^3$.  Then we compute 
$
\alpha\vert_{\mathit{HL}} = df 
$
where 
\begin{equation}
\label{eq:HL-primitive}
f = \frac{1}{4} r^2 (\sin(2s) + \sin(2t) - \sin(2s+2t))
\end{equation}
Note that $f = 0$ precisely when $(x_1,x_2,x_3)$ are in walls of the tetrahedral foam.

\begin{remark}
\label{rem:the-function-f}
It is possible to describe $\mathit{HL}$ as the graph of a double-valued one-form, indeed as the graph of $df$ where 
\[
f = \pm \frac{1}{2}\sum_{i = 1}^3 \epsilon_i x_i \sqrt{r^2 - x_i^2}
\] 
where
\begin{itemize}
\item $\epsilon_i = \epsilon_i(x_1,x_2,x_3) \in \{-1,1\}$ is a sign that depends on which chamber of the foam $(x_1,x_2,x_3)$ belongs to.  (Thus, let $\sqrt{r^2 -x_i^2}$ always denote the positive square root, and let this $\epsilon$ and the $\pm$ in front of the summation sign do the work recording the sign ambiguity of $f$).
\item $r$ is the largest solution to 
\[
r^3 - (x_1^2 + x_2^2 + x_3^2) r + 2x_1x_2x_3 = 0
\]
Note for all $(x_1,x_2,x_3) \in \bR^3$ all roots of that cubic equation are real --- the largest root is also the unique solution with $r \geq \max(|x_1|,|x_2|,|x_3|)$, and it has multiplicity $2$ exactly along the rays \eqref{eq:those-rays}.
\end{itemize}
\end{remark}

\subsection{Singular exact fillings}
\label{subsec:singular-exact-fillings}

Let $\Gamma \subset S^2$ be a cubic planar graph and let $\bF \subset D^3$ be a foam whose boundary is $\Gamma$.  Let $\pi:L(\bF) \to D^3$ denote the connected double cover branched over the $1$-skeleton of $\bF$, and write $\iota_{L(\bF)}$ for its nontrivial Deck transformation.  Then $L(\bF)$ is a manifold away from the vertices of $\bF$ --- near each vertex it is diffeomorphic to the Harvey-Lawson cone.  The boundary of $L(\bF)$ is naturally identified with the hyperelliptic surface $S$ branched over the vertices of $\Gamma$.

We choose a function $z:L(\bF) \to \bR$ that is smooth away from the cone points of $L(\bF)$, that is odd for the Deck transformation of $\pi$ (i.e. $z \circ \iota = -z$) and that has the following additional properties:
\begin{enumerate}
\item In a neighborhood $U$ of each cone point, there are coordinates on $\pi^{-1}(U)$ such that $z\vert_U$ looks like \eqref{eq:HL-primitive} 
\item The only critical points of $z$ are on $\pi^{-1}(\bF^1)$, i.e. on the critical points of $\pi$.  Furthermore, in a collar neighborhood of $\partial D^3$, $z$ is linear along each ray.
\end{enumerate}

Then (just as in \S\ref{subsec:EHK}, but one dimension up) we obtain an exact Lagrangian in $T^* D^3$ by taking $(\pi,\nabla z)$ whose restriction to $T^* D^3\vert_{\partial D^3} \cong (T^* S^2) \times \bR$ is Legendrian.  For example, the function $f$ of Remark \ref{rem:the-function-f} has almost all of these properties, except that it is quadratic instead of linear along each ray.  We obtain a function $z$ by damping $f$ with an exponent that smoothly interpolates between $1$ and $1/2$.

\begin{remark}
Note that item (1) of the conditions on $z$ is much stronger than the analogous item (4) of Definition \ref{def:hyperelliptic-wavefront}.  In particular it is only possible when there are coordinates near the vertex of $\bF$ that ``linearize'' the foam.  We suspect that it is possible to assume $\bF$ is in such a form without affecting the Hamiltonian isotopy class of the exact Lagrangian $L(\bF)$, but it would be desirable to relax this condition for other reasons.
\end{remark}

\begin{remark}
\label{rem:Sigma-A}
In the case of the Harvey-Lawson cone, this construction yields what Nadler denotes by ``$\Sigma \subset A$'' in the proof of \cite[Thm 3.2]{N-cone} --- $A$ is a copy of $\bR^3$
and $\Sigma$ is the hyperelliptic wavefront of Figure \ref{fig:cloudytet}.
\end{remark}

\subsection{Cobordisms}
\label{sec:cobordisms}
An exact filling of $\leg$ is an exact Lagrangian cobordism from $\leg$ to the empty Legendrian.  One can ask whether these cobordisms can be factored into simpler cobordisms --- as we discussed in \S\ref{subsec:EHK}, this is how Ekholm, Honda, and K\'alm\'an discovered their Lagrangians.  In this section we indicate a (somewhat) analogous description of the singular exact fillings of \S\ref{subsec:singular-exact-fillings}, in particular by rough analogy with the discussion of \S\ref{subsubsec:223}, but the analogy highlights some differences.

Suppose that $\Gamma$ has no self-loops or multiple edges, so its planar dual graph $\hat{\Gamma}$ is a triangulation of $S^2$.  Just as a cubic planar tree is dual to an ideal triangulation of a disk, an ideal foam is dual to an ideal ``tetahedronation'' of the ball $D^3$ that restricts to $\hat{\Gamma}$ on $\partial D^3 = S^2$.  Thus if $\bF$ is a foam filling of $\Gamma$, let $\hat{\bF}$ be the dual tetrahedronation restricting to $\hat{\Gamma}$ on the boundary.  If $\bF$ is an ideal foam, $\hat{\bF}$ is an ideal triangulation: it has no vertices besides those on $\partial D^3$.  For simplicity, we assume $\hat{\bF}$ is shellable, but as in the two-dimensional case (see Remark \ref{rmk:ehk}) this assumption is
unnecessary and $L(\bF)$ depends only on the foam $\bF$.

In \S\ref{subsubsec:223} we noted that a shelling of the triangles in the dual triangulation of a disk gives an ``EHK factorization'' of the exact Lagrangian surfaces into elementary cobordisms.  Thus, let $n$ be the number of tetrahedra in $\hat{\bF}$, equivalently the number of interior vertices in $\bF$, and let $a_n,a_{n-1},\ldots, a_1$ be a shelling of the tetrahedra in $\hat{\bF}$.  (We use descending indices to better match the notation of \S\ref{subsubsec:223}; the highest index is the outer shell.)  For each $k,$ the union of $a_n,\ldots,a_{n-k}$ is another polyhedron with a tetrahedronation --- it is dual to a foam filling $\bF(k)$ with boundary a cubic planar graph $\Gamma(k)$.  Write $L(\bF(k))$ for the
singular exact Lagrangian double cover branched over the one-skeleton of $\bF(k)$.  We will indicate how $L(\bF(k+1))$ is obtained from $L(\bF(k))$ and a Harvey-Lawson-type Lagrangian $\mathit{HL}_{k+1}$ by gluing $L(\bF(k))$ and $\mathit{HL}_{k+1}$ along parts of their boundaries.

For $\mathit{HL}_{k+1}$, we take the branched double cover of the one-skeleton of the foam associated with a small tetrahedron with one vertex on each face of $a_{n-(k+1)}$.  There are two cases to consider: the tetrahedron $a_{n-(k+1)}$ is incident with $\bigcup_{i = 1}^k a_{n-i}$ along either exactly one or exactly two of its faces.  (Since we assume $\hat{\bF}$ has no internal vertices, it is impossible for $a_{n-(k+1)}$ to be indicent with the previous simplices along three of its faces.)  

If it is incident along exactly one face, the left-hand figure illustrates $L(\bF(k)) \cup \mathit{HL}_{k+1}$ before the gluing, and the right-hand figure illustrates it after the gluing:
%
%
%
%
%
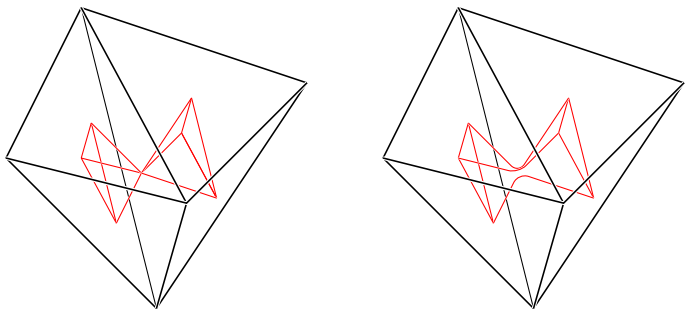
\begin{figure}[H]
\begin{tikzpicture}[scale=2]
\tikzset{%
  channel/.style    = { white, double = black, line width = .2pt,
                     double distance = 0.4pt },
  link/.style = { white, double = black, line width = 0.4pt,
                     double distance = 0.6pt },}
 \pgfmathsetmacro{\Ax}{0}
 \pgfmathsetmacro{\Ay}{0}
 \pgfmathsetmacro{\Bx}{1}
 \pgfmathsetmacro{\By}{-1}
 \pgfmathsetmacro{\Cx}{2}
 \pgfmathsetmacro{\Cy}{.5}
 \pgfmathsetmacro{\Dx}{.5}
 \pgfmathsetmacro{\Dy}{1}
 \pgfmathsetmacro{\Ex}{1.2}
 \pgfmathsetmacro{\Ey}{-.3}
 \pgfmathsetmacro{\ADEx}{ (1/3)*(\Ax+\Dx+\Ex)}
 \pgfmathsetmacro{\ADEy}{ (1/3)*(\Ay+\Dy+\Ey)}
 \pgfmathsetmacro{\ABEx}{(1/3)*(\Ax+\Bx+\Ex)}
 \pgfmathsetmacro{\ABEy}{(1/3)*(\Ay+\By+\Ey)}
 \pgfmathsetmacro{\ABDx}{(1/3)*(\Ax+\Bx+\Dx)}
 \pgfmathsetmacro{\ABDy}{(1/3)*(\Ay+\By+\Dy)}
 \pgfmathsetmacro{\BDEx}{(1/3)*(\Bx+\Dx+\Ex)}
 \pgfmathsetmacro{\BDEy}{(1/3)*(\By+\Dy+\Ey)}
 \pgfmathsetmacro{\BCDx}{(1/3)*(\Bx+\Cx+\Dx)}
 \pgfmathsetmacro{\BCDy}{(1/3)*(\By+\Cy+\Dy)}
 \pgfmathsetmacro{\BCEx}{(1/3)*(\Bx+\Cx+\Ex)}
 \pgfmathsetmacro{\BCEy}{(1/3)*(\By+\Cy+\Ey)}
 \pgfmathsetmacro{\CDEx}{(1/3)*(\Cx+\Dx+\Ex)}
 \pgfmathsetmacro{\CDEy}{(1/3)*(\Cy+\Dy+\Ey)}
 \pgfmathsetmacro{\thAx}{.5*\Ax}
 \def\Ay{0}
 \pgfmathsetmacro{\thAx}{\Ax/3}
 \coordinate (A) at (\Ax,\Ay);
 \coordinate (B) at (\Bx,\By);
 \coordinate (C) at (\Cx,\Cy);
 \coordinate (D) at (\Dx,\Dy);
 \coordinate (E) at (\Ex,\Ey);
 \coordinate (ADE) at (\ADEx,\ADEy);
 \coordinate (ABE) at (\ABEx,\ABEy);
 \coordinate (ABD) at (\ABDx,\ABDy);
 \coordinate (BDE) at (\BDEx,\BDEy);
 \coordinate (BCD) at (\BCDx,\BCDy);
 \coordinate (BCE) at (\BCEx,\BCEy);
 \coordinate (CDE) at (\CDEx,\CDEy);
 \draw[black,line width = .4pt] (B)--(D);
 \draw[red] (ADE)--(ABE);
 \draw[red] (ABD)--(ABD);
\draw[channel,double=red]  (ABD) --(BDE);
\draw[channel,double=red] (BDE)--(ADE);
\draw[red] (ADE)--(ABD);
\draw[red] (ABD)--(ABE);
\draw[channel,double = red]  (ABE)--(BDE);
 \draw[red] (CDE)--(BCE);
 \draw[red]  (BCE)--(BCD);
 \draw[channel,double=red] (BCD)--(BDE);
 \draw[channel,double=red] (BDE)--(CDE);
 \draw[red] (CDE)--(BCD);
 \draw[red] (BCD)--(BCE);
 \draw[channel,double=red] (BCE)--(BDE);
 \draw[black,line width = .6pt] (A)--(B);
 \draw[black,line width = .6pt] (A)--(D);
 \draw[black,line width = .6pt] (A)--(E);
 \draw[black,line width = .6pt] (B)--(C);
 \draw[black,line width = .6pt] (C)--(D);
 \draw[black,line width = .6pt] (C)--(E);
 \draw[link] (B)--(E)--(D);
 \draw[link] (A)--(E)--(C);
 \end{tikzpicture}
 \qquad
\begin{tikzpicture}[scale=2]
 node/.style={very thin}
\tikzset{%
  channel/.style    = { white, double = black, line width = .2pt,
                     double distance = 0.4pt },
  link/.style = { white, double = black, line width = 0.4pt,
                     double distance = 0.6pt },}
 \pgfmathsetmacro{\Ax}{0}
 \pgfmathsetmacro{\Ay}{0}
 \pgfmathsetmacro{\Bx}{1}
 \pgfmathsetmacro{\By}{-1}
 \pgfmathsetmacro{\Cx}{2}
 \pgfmathsetmacro{\Cy}{.5}
 \pgfmathsetmacro{\Dx}{.5}
 \pgfmathsetmacro{\Dy}{1}
 \pgfmathsetmacro{\Ex}{1.2}
 \pgfmathsetmacro{\Ey}{-.3}
 \pgfmathsetmacro{\ADEx}{ (1/3)*(\Ax+\Dx+\Ex)}
 \pgfmathsetmacro{\ADEy}{ (1/3)*(\Ay+\Dy+\Ey)}
 \pgfmathsetmacro{\ABEx}{(1/3)*(\Ax+\Bx+\Ex)}
 \pgfmathsetmacro{\ABEy}{(1/3)*(\Ay+\By+\Ey)}
 \pgfmathsetmacro{\ABDx}{(1/3)*(\Ax+\Bx+\Dx)}
 \pgfmathsetmacro{\ABDy}{(1/3)*(\Ay+\By+\Dy)}
 \pgfmathsetmacro{\BDEx}{(1/3)*(\Bx+\Dx+\Ex)}
 \pgfmathsetmacro{\BDEy}{(1/3)*(\By+\Dy+\Ey)}
 \pgfmathsetmacro{\BCDx}{(1/3)*(\Bx+\Cx+\Dx)}
 \pgfmathsetmacro{\BCDy}{(1/3)*(\By+\Cy+\Dy)}
 \pgfmathsetmacro{\BCEx}{(1/3)*(\Bx+\Cx+\Ex)}
 \pgfmathsetmacro{\BCEy}{(1/3)*(\By+\Cy+\Ey)}
 \pgfmathsetmacro{\CDEx}{(1/3)*(\Cx+\Dx+\Ex)}
 \pgfmathsetmacro{\CDEy}{(1/3)*(\Cy+\Dy+\Ey)}
 \pgfmathsetmacro{\thAx}{.5*\Ax}
 \def\Ay{0}
 \pgfmathsetmacro{\thAx}{\Ax/3}
 \coordinate (A) at (\Ax,\Ay);
 \coordinate (B) at (\Bx,\By);
 \coordinate (C) at (\Cx,\Cy);
 \coordinate (D) at (\Dx,\Dy);
 \coordinate (E) at (\Ex,\Ey);
 \coordinate (ADE) at (\ADEx,\ADEy);
 \coordinate (ABE) at (\ABEx,\ABEy);
 \coordinate (ABD) at (\ABDx,\ABDy);
 \coordinate (BDE) at (\BDEx,\BDEy);
 \coordinate (BCD) at (\BCDx,\BCDy);
 \coordinate (BCE) at (\BCEx,\BCEy);
 \coordinate (CDE) at (\CDEx,\CDEy);
 \draw[black,line width = .4pt] (B)--(D);
 \draw[red] (ADE)--(ABE);
 \draw[red] (ABE)--(ABD);
 \draw[red] (ABD)--(ADE);
 \draw[red] (CDE)--(BCE);
 \draw[red] (BCE)--(BCD);
 \draw[red] (BCD)--(CDE);
 \draw[channel,double=red,rounded corners] (ADE)--(BDE)--(CDE);
 \draw[channel,double=red,rounded corners] (ABE)--(BDE)--(BCE);
 \draw[channel,double=red,rounded corners] (ABD)--(BDE)--(BCD);
 \draw[black,line width = .6pt] (A)--(B);
 \draw[black,line width = .6pt] (A)--(D);
 \draw[black,line width = .6pt] (A)--(E);
 \draw[black,line width = .6pt] (B)--(C);
 \draw[black,line width = .6pt] (C)--(D);
 \draw[black,line width = .6pt] (C)--(E);
 \draw[link] (B)--(E)--(D);
 \draw[link] (A)--(E)--(C);
 \end{tikzpicture}
\caption{The black tetrahedra are $a_n$ and $a_{n-1}$.  On the left, one of the red tetrahedra is the boundary of $\bF(1)$ and one of them is the boundary of the foam associated with $\mathit{HL}_2$.
On the right, the red figure indicates the boundary of $\bF(2)$ obtained by gluing $\mathit{HL}_2$ to $\bF(1).$}
\label{fig:facecobordism}
\end{figure}
\noindent More precisely, if $a_{n-(k+1)}$ is incident with the previous simplices along exactly one face, then $L(\bF(k+1))$ is obtained from $L(\bF(k))$ and $\mathit{HL}_{k+1}$ by
modifying the disconnected wavefront near the two vertices meeting at the interior face.
Specifically, remove from each wavefront a disk neighborhood of the relevant vertex and replace these disks by gluing in a cylinder, as in Figure \ref{fig:facecobordism}.


%

To understand the case of an interior face of ${\bf F}$, i.e.~edge of $\hat{\bf F}$, consider Figure \ref{fig:edgecobordism} below.
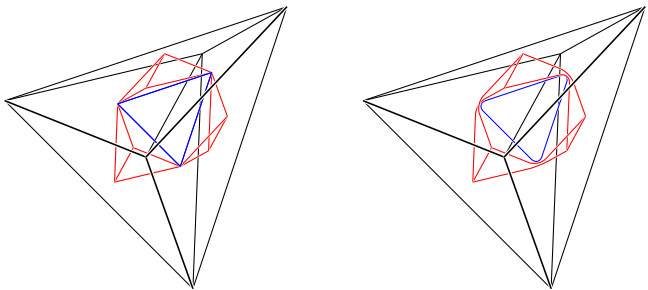
\begin{figure}[H]
\begin{tikzpicture}
 node/.style={very thin}
\tikzset{%
  channel/.style    = { white, double = black, line width = .2pt,
                     double distance = 0.4pt },
  link/.style = { white, double = black, line width = 0.4pt,
                     double distance = 0.6pt },}
 \pgfmathsetmacro{\Ax}{-2.5}
 \pgfmathsetmacro{\Ay}{0}
 \pgfmathsetmacro{\Bx}{0}
 \pgfmathsetmacro{\By}{-2.5}
 \pgfmathsetmacro{\Cx}{1.25}
 \pgfmathsetmacro{\Cy}{1.25}
 \pgfmathsetmacro{\Dx}{-.625}
 \pgfmathsetmacro{\Dy}{-.75}
 \pgfmathsetmacro{\Ex}{.125}
 \pgfmathsetmacro{\Ey}{.625}
 \pgfmathsetmacro{\ADEx}{ (1/3)*(\Ax+\Dx+\Ex)}
 \pgfmathsetmacro{\ADEy}{ (1/3)*(\Ay+\Dy+\Ey)}
 \pgfmathsetmacro{\ABEx}{(1/3)*(\Ax+\Bx+\Ex)}
 \pgfmathsetmacro{\ABEy}{(1/3)*(\Ay+\By+\Ey)}
 \pgfmathsetmacro{\ABDx}{(1/3)*(\Ax+\Bx+\Dx)}
 \pgfmathsetmacro{\ABDy}{(1/3)*(\Ay+\By+\Dy)}
 \pgfmathsetmacro{\BDEx}{(1/3)*(\Bx+\Dx+\Ex)}
 \pgfmathsetmacro{\BDEy}{(1/3)*(\By+\Dy+\Ey)}
 \pgfmathsetmacro{\BCDx}{(1/3)*(\Bx+\Cx+\Dx)}
 \pgfmathsetmacro{\BCDy}{(1/3)*(\By+\Cy+\Dy)}
 \pgfmathsetmacro{\BCEx}{(1/3)*(\Bx+\Cx+\Ex)}
 \pgfmathsetmacro{\BCEy}{(1/3)*(\By+\Cy+\Ey)}
 \pgfmathsetmacro{\CDEx}{(1/3)*(\Cx+\Dx+\Ex)}
 \pgfmathsetmacro{\CDEy}{(1/3)*(\Cy+\Dy+\Ey)}
 \pgfmathsetmacro{\ABCx}{(1/3)*(\Ax+\Bx+\Cx)}
 \pgfmathsetmacro{\ABCy}{(1/3)*(\Ay+\By+\Cy)}
 \pgfmathsetmacro{\ACDx}{(1/3)*(\Ax+\Cx+\Dx)}
 \pgfmathsetmacro{\ACDy}{(1/3)*(\Ay+\Cy+\Dy)}
 \pgfmathsetmacro{\ACEx}{(1/3)*(\Ax+\Cx+\Ex)}
 \pgfmathsetmacro{\ACEy}{(1/3)*(\Ay+\Cy+\Ey)}
 \pgfmathsetmacro{\ADEBDEx}{ (1/2)*(\ADEx+\BDEx)}
 \pgfmathsetmacro{\ADEBDEy}{ (1/2)*(\ADEy+\BDEy)}
 \pgfmathsetmacro{\ABDADEx}{ (1/2)*(\ABDx+\ADEx)}
 \pgfmathsetmacro{\ABDADEy}{ (1/2)*(\ABDy+\ADEy)}
 \coordinate (A) at (\Ax,\Ay);
 \coordinate (B) at (\Bx,\By);
 \coordinate (C) at (\Cx,\Cy);
 \coordinate (D) at (\Dx,\Dy);
 \coordinate (E) at (\Ex,\Ey);
 \coordinate (ADE) at (\ADEx,\ADEy);
 \coordinate (ABE) at (\ABEx,\ABEy);
 \coordinate (ABD) at (\ABDx,\ABDy);
 \coordinate (BDE) at (\BDEx,\BDEy);
 \coordinate (BCD) at (\BCDx,\BCDy);
 \coordinate (BCE) at (\BCEx,\BCEy);
 \coordinate (CDE) at (\CDEx,\CDEy);
 \coordinate (ABC) at (\ABCx,\ABCy);
 \coordinate (ACD) at (\ACDx,\ACDy);
 \coordinate (ACE) at (\ACEx,\ACEy);
 \coordinate (ADEBDE) at (\ADEBDEx,\ADEBDEy);
 \coordinate (ABDADE) at (\ABDADEx,\ABDADEy);
 \draw[black,line width = .4pt] (A)--(B);
 \draw[black,line width = .4pt] (A)--(C);
 \draw[black,line width = .4pt] (A)--(E);
 \draw[black,line width = .4pt] (B)--(C);
 \draw[black,line width = .4pt] (B)--(E);
 \draw[black,line width = .4pt] (C)--(E);
 \draw[red] (ABD)--(ABE);
 \draw[red] (ABD)--(ADE);
 \draw[red] (ABD)--(BDE);
 \draw[red] (ABE)--(ADE);
 \draw[red] (ABE)--(BDE);
 %
 \draw[black,line width = .4pt] (D)--(E);
 \draw[channel,double=red] (ABD)--(ABE);
 \draw[channel,double=red] (ABD)--(ADE);
 \draw[channel,double=red] (ABD)--(BDE);
 \draw[channel,double=red] (ABE)--(ADE);

 %
 \draw[channel,double=red] (ACD)--(ACE);
 \draw[channel,double=red] (ACD)--(ADE);
 \draw[channel,double=red] (ACD)--(CDE);
 \draw[channel,double=red] (ACE)--(ADE);
 \draw[channel,double=red] (ACE)--(CDE);
 %
 \draw[channel,double=red] (BCD)--(BCE);
 \draw[channel,double=red] (BCD)--(BDE);
 \draw[channel,double=red] (BCD)--(CDE);
 \draw[channel,double=red] (BCE)--(BDE);
 \draw[channel,double=red] (BCE)--(CDE);
  \draw[channel,double=blue] (ADE)--(BDE);
   \draw[channel,double=blue] (BDE)--(CDE);
    \draw[channel,double=blue] (CDE)--(ADE);
 \draw[blue] (ADE)--(BDE);
   \draw[blue] (BDE)--(CDE);
    \draw[blue] (CDE)--(ADE);
 \draw[link] (D)--(A);
 \draw[link] (D)--(B);
  \draw[link] (D)--(C);
 \end{tikzpicture}
 \qquad
\begin{tikzpicture}
 node/.style={very thin}
\tikzset{%
  channel/.style    = { white, double = black, line width = .2pt,
                     double distance = 0.4pt },
  link/.style = { white, double = black, line width = 0.4pt,
                     double distance = 0.6pt },}
 \pgfmathsetmacro{\Ax}{-2.5}
 \pgfmathsetmacro{\Ay}{0}
 \pgfmathsetmacro{\Bx}{0}
 \pgfmathsetmacro{\By}{-2.5}
 \pgfmathsetmacro{\Cx}{1.25}
 \pgfmathsetmacro{\Cy}{1.25}
 \pgfmathsetmacro{\Dx}{-.625}
 \pgfmathsetmacro{\Dy}{-.75}
 \pgfmathsetmacro{\Ex}{.125}
 \pgfmathsetmacro{\Ey}{.625}
 \pgfmathsetmacro{\ADEx}{ (1/3)*(\Ax+\Dx+\Ex)}
 \pgfmathsetmacro{\ADEy}{ (1/3)*(\Ay+\Dy+\Ey)}
 \pgfmathsetmacro{\ABEx}{(1/3)*(\Ax+\Bx+\Ex)}
 \pgfmathsetmacro{\ABEy}{(1/3)*(\Ay+\By+\Ey)}
 \pgfmathsetmacro{\ABDx}{(1/3)*(\Ax+\Bx+\Dx)}
 \pgfmathsetmacro{\ABDy}{(1/3)*(\Ay+\By+\Dy)}
 \pgfmathsetmacro{\BDEx}{(1/3)*(\Bx+\Dx+\Ex)}
 \pgfmathsetmacro{\BDEy}{(1/3)*(\By+\Dy+\Ey)}
 \pgfmathsetmacro{\BCDx}{(1/3)*(\Bx+\Cx+\Dx)}
 \pgfmathsetmacro{\BCDy}{(1/3)*(\By+\Cy+\Dy)}
 \pgfmathsetmacro{\BCEx}{(1/3)*(\Bx+\Cx+\Ex)}
 \pgfmathsetmacro{\BCEy}{(1/3)*(\By+\Cy+\Ey)}
 \pgfmathsetmacro{\CDEx}{(1/3)*(\Cx+\Dx+\Ex)}
 \pgfmathsetmacro{\CDEy}{(1/3)*(\Cy+\Dy+\Ey)}
 \pgfmathsetmacro{\ABCx}{(1/3)*(\Ax+\Bx+\Cx)}
 \pgfmathsetmacro{\ABCy}{(1/3)*(\Ay+\By+\Cy)}
 \pgfmathsetmacro{\ACDx}{(1/3)*(\Ax+\Cx+\Dx)}
 \pgfmathsetmacro{\ACDy}{(1/3)*(\Ay+\Cy+\Dy)}
 \pgfmathsetmacro{\ACEx}{(1/3)*(\Ax+\Cx+\Ex)}
 \pgfmathsetmacro{\ACEy}{(1/3)*(\Ay+\Cy+\Ey)}
 \pgfmathsetmacro{\ADEBDEx}{ (1/2)*(\ADEx+\BDEx)}
 \pgfmathsetmacro{\ADEBDEy}{ (1/2)*(\ADEy+\BDEy)}
 \pgfmathsetmacro{\ABDADEx}{ (1/2)*(\ABDx+\ADEx)}
 \pgfmathsetmacro{\ABDADEy}{ (1/2)*(\ABDy+\ADEy)}
 \coordinate (A) at (\Ax,\Ay);
 \coordinate (B) at (\Bx,\By);
 \coordinate (C) at (\Cx,\Cy);
 \coordinate (D) at (\Dx,\Dy);
 \coordinate (E) at (\Ex,\Ey);
 \coordinate (ADE) at (\ADEx,\ADEy);
 \coordinate (ABE) at (\ABEx,\ABEy);
 \coordinate (ABD) at (\ABDx,\ABDy);
 \coordinate (BDE) at (\BDEx,\BDEy);
 \coordinate (BCD) at (\BCDx,\BCDy);
 \coordinate (BCE) at (\BCEx,\BCEy);
 \coordinate (CDE) at (\CDEx,\CDEy);
 \coordinate (ABC) at (\ABCx,\ABCy);
 \coordinate (ACD) at (\ACDx,\ACDy);
 \coordinate (ACE) at (\ACEx,\ACEy);
 \coordinate (ADEBDE) at (\ADEBDEx,\ADEBDEy);
 \coordinate (ABDADE) at (\ABDADEx,\ABDADEy);
 \draw[black,line width = .4pt] (A)--(B);
 \draw[black,line width = .4pt] (A)--(C);
 \draw[black,line width = .4pt] (A)--(E);
 \draw[black,line width = .4pt] (B)--(C);
 \draw[black,line width = .4pt] (B)--(E);
 \draw[black,line width = .4pt] (C)--(E);
 %
 \draw[black,line width = .4pt] (D)--(E);
 \draw[channel,double=red] (ABD)--(ABE);
 \draw[channel,double=blue,rounded corners] (ADEBDE)--(BDE)--(CDE)--(ADE)--(ADEBDE);
 \draw[channel,double=red,rounded corners] (ABD)--(ADE)--(ACD);
 \draw[channel,double=red,rounded corners] (ACD)--(CDE)--(BCD);
 \draw[channel,double=red,rounded corners] (BCD)--(BDE)--(ABD);
 \draw[channel,double=red,rounded corners] (ABE)--(BDE)--(BCE);
 \draw[channel,double=red,rounded corners] (BCE)--(CDE)--(ACE);
 \draw[channel,double=red,rounded corners] (ACE)--(ADE)--(ABE);
 \draw[channel,double=red] (ACD)--(ACE);
 %
 \draw[channel,double=red] (BCD)--(BCE);
 \draw[link] (D)--(A);
 \draw[link] (D)--(B);
  \draw[link] (D)--(C);
 \end{tikzpicture}
\caption{In case $a_{n-(k+1)}$ shares a face with two tetrahedra of $\hat{\bF}(k)$, the red+blue graph that
emerges from the cobordism is not the boundary of $\bF(k+1)$, but of a disjoint union of $\bF(k+1)$ and a spurious
(blue) loop.}
\label{fig:edgecobordism}
\end{figure}

To define $\bF(k+1)$ and obtain from it a singular exact filliing, we construct a cobordism from the boundary of $L(\bF(k))\sqcup HL_{k+1}$ to a new space
by modifying the wavefront to remove the spurious blue loop.
Consider a neighborhood of the blue loop in the front projection.  It is a cylinder.
The cobordism takes this cylinder to two disks, just like the two-to-one-sheet hyperboloid cobordism $x^2 + y^2 - z^2 = t,\; t\in [0,1]$.
Afterward, the blue loop disappears and the red edges make up a new polytope, or dually define $\bF(k+1)$, which in Figure \ref{fig:edgecobordism} is a triangular prism.

%

\subsection{The tangles associated to a foam}
\label{subsec:tangles}

The one-skeleton $\bF^1$ of a foam $\bF \subset D^3$ is a kind of singular tangle in the ball, joining the vertices of graph on the boundary.  The singularities are at the internal vertices $\bF^0$ of the foam, suppose that there are $T$ of these.  Then we can associated as many as $3^T$ smoothings of this singular tangle, as follows.  In a neighborhood of each of the $T$ vertices, choose coordinates in which $\bF^1$ looks like the rays of \eqref{eq:those-rays}, and replace $\bF^1$ by one of the smooth hyperbolas
\begin{equation}
\label{eq:smoothed-tangle}
\{(x_1,x_2,x_3) \mid x_{i_2} = x_{i_3} \text{ and } x_{i_1}^2 - x_{i_2}x_{i_3} = 1\}
\end{equation}
where $i_1,i_2,i_3$ is a permutation of $\{1,2,3\}$ --- the hyperbola is determined just by the value of $i_1$.
In fact, this is precisely the critical locus of the projection to $\bR^3_x$ of the Harvey-Lawson smoothing,
the solid torus of Example \ref{ex:the-HL-solid} below.

We have not investigated what kinds of tangles can appear, the examples we have encountered so far are all quite unlinked.

\begin{remark}
In \cite{CCV} and especially \cite{CEHRV}, singular and smooth tangles were considered in a nearly identical context:  constructing
Lagrangian branes.  However, the authors there did not consider fillings of a fixed Legendrian boundary.  Our purposes require the
consideration of foams.  In the work of \cite{CEHRV},
all smoothings of tangles were considered.  We do not know if such tangles occur within the context of foams.
\end{remark}

\begin{example}
Let $\bF$ be the foam on the left and $\bG$ the foam on the right in Figure \ref{fig:two-foams}.  
In general, smoothings are naturally indexed by the data of, at each internal vertex $v$, a partition of the four edges incident with $v$ into two pairs.  Taken up to the dihedral symmetry of the foams, there are three possibilities for $\bF$ and six for $\bG$.  (For $\bF$, one of these smoothings is pictured in blue in \S\ref{sec:tentframing}.)  Eight of these nine tangles are abstractly homeomorphic to each other, and in fact to the ``trivial tangle'' of three parallel strands.  The ninth tangle, which appears for $\bG$: it is the union of three parallel strands, along with an unknot which is not linked with any strand.
\end{example}

\subsection{Nonexact fillings --- tangles as caustics}
\label{subsec:nonexact-fillings}

We have constructed singular exact fillings of $\leg$ with the topology of double covers of the ball, branched over $\bF^1$, the edges of a foam.  Example \ref{ex:the-HL-solid} below suggests that we search for smooth fillings with the topology of a double cover of the ball, branched over a tangle obtained by smoothing $\bF^1$ in the sense of \S\ref{subsec:tangles}.  However the example also suggests that such smoothings will not be exact.  Theorem \ref{thm:nofillings} gives a strong result making this precise.  Let us therefore make some general remarks about nonexact fillings.

Let $(M,\omega = d \alpha)$ be an exact symplectic manifold.  In the literature on Lagrangian fillings of Legendrians $\Lambda \subset \partial M$, the condition that the Lagrangian has conical ends is often imposed.  For 3-dimensional Lagrangians, this has an undesirable consequence.  If $L$ is $3$-dimensional and $U$ is a collar neighborhood of the boundary of $L$, the map $H^1(L,\bR) \to H^1(U,\bR)$ is often injective, so if $\alpha$ is exact on $U$, it is exact on all of $L$.  This suggests that if we wish to consider non-exact Lagrangians, the condition that $L$ has conical ends is inappropriate.  A notion of ``asymptotically conical'' might be more appropriate, as in \cite[Def. 7.1]{J} for special Lagrangians, or \cite[Def. 5.4.1]{NZ} for subanalytic Lagrangians.

\begin{example}[The Harvey-Lawson solid tori]
\label{ex:the-HL-solid}
There are three one-parameter families of these, given parametrically by
\[
(r,e^{is},e^{it}) \mapsto (\sqrt{r^2 + \epsilon^2} e^{is}, re^{it}, r e^{-is-it})
\]
and its permutations.  The antiderivative of $\alpha\vert_{\mathit{HL}_{\epsilon}}$ is
\[
\text{$f$ of \eqref{eq:HL-primitive}} + \frac{1}{2} \epsilon^2 \sin(s) \cos(t) + \frac{1}{2} \epsilon^2 s
\]
The last term, $\frac{1}{2} \epsilon^2 s$, is not periodic and therefore $\mathit{HL}_{\epsilon}$ is not exact --- however it also has no $r$-dependence.
It is an asymptotically conic special Lagrangian, in the sense studied by Joyce (see \cite[Example 6.9]{J}).
For general tangles, the existence of asymptotically special Lagrangian smooth embeddings in $\bC^3$ is unproven.
However, we have the following construction of non-special Lagrangian fillings in $\bC^3$, potentially with isolated
immersed double points.

\end{example}

If $\cT$ is a tangle associated to a foam $\bF$, write $\pi:L(\cT) \to D^3$ for the double cover branched over $\cT$.  In the rest of the paper we will refer to any Lagrangian map $L(\cT) \to T^* D^3$ whose projection to $D^3$ is $\pi$ as a ``filling'' of $\leg$ --- or a ``filling modeled on $\cT$'' --- without heed for asymptotic conditions.
Note that even so, the boundary of $L(\cT)$ is canonically identified with $\leg$,  since $\cT \cap \partial D^3$ is exactly the set of vertices of $\Gamma$.  We also will not necessarily assume that the map $L(\cT) \to T^* D^3$ is an embedding.

Once a filling modeled on $\cT$ is given, the pull-back of the canonical one-form on $T^* D^3$ to $L(\cT)$ is closed.  Conversely, for each closed one-form there is at most one Lagrangian map $L(\cT) \to T^* D^3$ whose projection to $D^3$ is $\pi$.

To see that a filling modeled on $\cT$ is determined by a $1$-form on $L(\cT)$, write $x_1,x_2,x_3$ for the coordinates on $D^3$, $y_1,y_2,y_3$ for the coordinates on $T^* D^3$, and let $s_1,s_2,s_3$ be local coordinates on $L(\cT)$.  
If $\alpha = \alpha_1 ds_1 + \alpha_2 ds_2 + \alpha_3 ds_3$ is a closed one-form on $L(\cT)$, then $y_1,y_2,y_3$ are the solutions to $\sum y_i dx_i = \alpha$, which is equivalent to
\begin{equation}
\label{eq:Lagrangian-equation}
\left(
\begin{array}{rrr}
{\partial x_1}/{\partial s_1} & {\partial x_2} /{\partial{s_1}} &  {\partial x_3}/{\partial s_1} \\
{\partial x_1}/{\partial s_2} & {\partial x_2} /{\partial{s_2}} &  {\partial x_3}/{\partial s_2} \\
{\partial x_1}/{\partial s_3} & {\partial x_2} /{\partial{s_3}} &  {\partial x_3}/{\partial s_3} \\
\end{array}
\right)\left(
\begin{array}{r}
y_1 \\ y_2 \\ y_3
\end{array}
\right) =
\left(\begin{array}{rrr}
\alpha_1 \\
\alpha_2 \\
\alpha_3
\end{array}
\right)
\end{equation}
The matrix on the left-hand side is the Jacobian of $\pi$, which is invertible away from $\cT$.  For \eqref{eq:Lagrangian-equation} to have a solution over $\cT$ requires that $\alpha$ vanish along $\cT$.  The existence of such $\alpha$ is addressed by the following:

\begin{proposition}
\label{prop:one-form}
Let $\pi: L(\cT) \to D^3$ be a branched double cover of a tangle $\cT \subset D^3$,
and let $\iota$ denote the Deck involution.
Suppose that $\cT$ has no circle components (just strands), and let $\theta$ be a closed one-form on $L(\cT)$.
Then there is a function $F:L(\cT)\to R$
such that $\alpha := \theta + dF$ is \emph{odd} under $\iota$, i.e. $\iota^* \alpha = -\alpha$.  In particular, $\alpha$ vanishes on $\cT$.  Moreover, any other $\alpha'$ satisfying these properties differs from $\alpha$ by $dG$, where
$G = -G \circ\iota$, i.e. $G$ is odd under $\iota.$
\end{proposition}

Before proving the proposition, we remark:  it follows that for each class in $H^1(L(\cT),\bR)$, there is a Lagrangian immersion into $\bR^6 \cong T^*D^3$ with at worst isolated double points, such that the pull-back of the
canonical form of $T^*D^3$ is in the same cohomology class.
Indeed we can select $\alpha$ as in Proposition \ref{prop:one-form} and construct the embedding as described by \eqref{eq:Lagrangian-equation}.
Zeroes of $\alpha$ give double points of the immersion.

\begin{proof}
Let us write $L$ for $L(\cT).$
If $Z^1(L,\cT)$ is the space of closed $1$-forms that vanish on $\cT$, and $B^1(L,\cT)$ is the subspace of $1$-forms like $dG$, where $G$ vanishes on $\cT$, then we have a short exact sequence of $\iota$-modules
\[
0 \to B^1(L,\cT) \to Z^1(L,\cT) \to H^1(L,\cT) \to 0
\]
We will prove the Proposition by proving that the tautological map $H^1(L,\cT) \to H^1(L)$ is an isomorphism when restricted to the $(-1)$-eigenspace of $H^1(L,\cT)$.

We have assumed $\cT$ has no circle components, so $H^1(\cT) = 0$.  The long exact sequence of the pair $(L,\cT)$ therefore induces a short exact sequence
\begin{equation}
\label{eq:SES-in-proof}
H^0(L) \to H^0(T) \to H^1(L,\cT) \to H^1(L) \to 0
\end{equation}
\eqref{eq:SES-in-proof} is a short exact sequence of $\iota$-modules.  Both  $H^0(\cT)/H^0(L)$ and $H^1(L)$, have dimension $g$ over $\bR$ --- $\iota$ acts trivially on $H^0(\cT)/H^0(L)$ and (since $L/\iota = D^3$ and $H^1(D^3) = 0$) by $-1$ on $H^1(L)$.  Thus the $(-1)$-eigenspace of the $\iota$-action on $H^1(L,\cT)$ is identified with $H^1(L)$.

\end{proof}

\begin{remark}
It is tempting to speculate that by choosing $G$ in the Proposition correctly, one can make the map $L(\cT) \to T^* D^3 \cong \bR^6$ a special Lagrangian immersion, or even embedding.  This amounts to showing the existence of a solution to a PDE of Monge-Amp\`ere type for the single function $G$.
\end{remark}

\section{Constructible sheaves}

To this point, we have defined a Legendrian surface $\leg$ from a cubic, planar graph $\Gamma,$ and singular, exact Lagrangian fillings
from foams, as well as their smoothings via tangles.  In this section, we study a category of sheaves defined by $\leg,$
which by \cite{N2,NZ} is a Fukaya category.  We give a concrete description of this category and its moduli space $\cM$ of simple objects,
then use this moduli space to give many examples of non-isotopic Legendrians with the same genus and classical invariants,
and to prove non-fillability for Legendrians constructed
from simple cubic planar graphs. 
We show that $\cM$ embeds into a period domain as a Lagrangian submanifold, and define a generalized notion of ``phase'' and ``framing'' (related
to the tangle description of the filling), to define conjectural open Gromov-Witten invariants \`a la \cite{AV,AKV}.
\medskip

Fix a commutative ground ring $k$ and write $\Sh(\bR^3)$ for the $k$-linear dg-derived category of constructible sheaves of $k$-modules on $\bR^3$.  In this section, we study full subcategories of $\Sh(\bR^3)$ with singular support determined by a hyperelliptic wavefront (Definition \ref{def:hyperelliptic-wavefront}) --- i.e. with singular support in either the Legendrian lift or the extended Legendrian lift (see Section \ref{subsec:extended-leg}) of a hyperelliptic wavefront.  Thus fix a graph $\Gamma \subset S^2$, let $\Psi \subset \bR^3$ be the image of the associated hyperelliptic wavefront, and let $\leg \subset \leg^+ \subset T^{\infty} \bR^3$ be the associated Legendrians.  We define 
\[
\cC(\leg) \subset \cC(\leg^+) \subset \Sh(\bR^3)
\]
to be the full subcategories of sheaves that are compactly supported, and that have singular support in $\leg$ or $\leg^+$.
That is, $\cC(\leg) := \Sh(\bR^3,\leg)$ and $\cC(\leg^+) := \Sh(\bR^3,\leg^+).$

\subsection{Regular cell decomposition}
\label{subsec:regular-cell-decomposition}

We let $g$ be as in \eqref{eq:fveg}, and write $\Psi$ for the image of the hyperelliptic wavefront.  The filtration 
\begin{equation}
\label{eq:stratification}
\text{(vertices of $\Gamma$)} \subset \Gamma \subset \Psi \subset \bR^3
\end{equation}
gives a Whitney stratification of $\bR^3$, with $2g+2$ strata of dimension zero, $3g+3$ strata of dimension one, $2g+6$ strata of dimension two and $g+5$ top-dimensional strata.  The edges and vertices of $\Gamma$ are also edges and vertices of \eqref{eq:stratification}, but let us give some vocabulary for the two- and three-dimensional strata.

For the top-dimensional strata, there is unique region containing the origin and a unique region incident with the point at $\infty$, that we call the ``inner'' and ``outer'' region respectively.  The remaining open strata are in one-to-one correspondence with the $g+3$ faces of $\Gamma$, we call these regions ``pillows.''  Each two-dimensional stratum is in the boundary of a unique pillow --- we call them ``sheets.''  Each pillow has exactly two sheets at its boundary, the ``inner sheet'' which is incident with the inner region and the ``outer sheet'' incident with the outer region.

If we omit the outer region (which has the topology of $S^2 \times \bR$), the stratification is a regular cell complex.  We therefore have an equivalence (see e.g.  \cite[\S 8.1]{KS} \cite[Prop. 3.9]{STZ}) between 
\begin{itemize}
\item the derived category of sheaves that are constructible for the stratification and acyclic in the outer region 
\item the derived category of functors from the partially ordered set of strata (not including the outer region) to $k$-modules.
\end{itemize}
We will use this equivalence freely in what follows, describing a sheaf by a strictly commutative diagram of chain complexes in the shape of the poset of strata.  It is convenient to include the outer region in these diagrams, but it is always to be labeled by the zero complex.  

We will describe conditions on such diagrams to belong to $\cC(\leg^+)$ and $\cC(\leg)$.  These conditions are local, so the cases to be considered are a neighborhood of a vertex, and edge, and an inner or outer sheet.  In the end we find that an object of $\cC(\leg^+)$ or of $\cC(\leg)$ has a very simple description, which we summarize in \S\ref{subsec:concrete}.

\subsubsection{Local study at one- and two-dimensional strata}
\label{subsec:local-study-one-two}
In codimension two or less, a wavefront hypersurface (of a manifold of any dimension $n$) looks like a product of a
smooth hypersurface in $\bR^{n-2}$ with the front diagram of a Legendrian knot in $\bR^3$ --- we have studied these in detail in \cite{STZ}.  We recall the local descriptions for a front without cusps (thus without any genuine front singularities) here --- it is a special case of \cite[Thm. 3.12]{STZ}.

Suppose $\Sigma$ is a sheet incident with the two regions $R_1$ and $R_2$, with $R_1$ farther from the origin than $R_2$.  Then a sheaf $F$ is given near $\Sigma$ by a diagram
\[
F(R_1) \leftarrow F(\Sigma) \to F(R_2)
\]
If $\Sigma$ is an outer sheet, then $R_1$ is the outer region and we require $F(R_1)$ to be the zero complex.  If the singular support is to lie in $\leg^+$ (or $\leg$, away from the edges there is no difference), the map $F(\Sigma) \to F(R_1)$ must be an isomorphism.

If $i$ is an edge of $\Gamma$, a sheaf $F$ is given near $i$ by a diagram
\begin{equation}
\label{eq:typical-TTT}
\xymatrix{
& & U  \\
& A \ar[ur] \ar[dl] & & A' \ar[ul] \ar[dr] \\
B & & \ar[ur] \ar[ul] \ar[dr] \ar[dl] X & & B' \\
& C \ar[dr] \ar[ul] & & C' \ar[ur] \ar[dl] \\
& & D
}
\end{equation}
where $U$ is the complex labeling the outer region (it must be zero), $B$ and $B'$ label the pillows, $D$ labels the inner region, $A$ and $A'$ label upper sheets, $C$ and $C'$ label lower sheets, and $X$ labels the edge.  The sheaf belongs to $\cC(S^+)$ if and only if the six maps $X \to A$, $X \to A'$, $A \to U$, $A' \to U$, $C \to B$, $C' \to B'$ are quasi-isomorphisms.  Note in particular that this requires that $X, A$, and $A'$ are acyclic.  It belongs to $\cC(S)$ if and only if the square at the bottom is exact (i.e. it realizes $X$ as the homotopy fiber product of the maps $C,C' \to D$).

\subsubsection{Local study at a vertex}
\label{subsec:local-study-vertex}

The natural stratification of the wavefront given by \eqref{eq:vertexequation} has one vertex, three edges, six 2-dimensional strata (wedge-shaped sheets), and five 3-dimensional strata (three pillows in between the sheets, an outer region and an inner region).  Their pattern of incidences is recorded in \eqref{eq:big-diagram} --- the vertex is denoted $Z$, the edges are denoted $X$, the sheets are the $A$'s and $C$'s, the three pillows are $B$s, and the big open regions are $U$ and $D$.

\begin{equation}
\label{eq:big-diagram}
\xymatrix{
& & U \\
& & & & A_1 \ar[llu] \ar[d] \\
& & & & B_1 \\
& X_2  \ar[dl] \ar[dddl] \ar[rrruu] \ar[rrr] & & & C_1 \ar[dddddl] \ar[u] &  \\
 A_2\ar[uuuurr] \ar[d] & & & Z \ar[rrr] \ar[dl] \ar[ull] & & & X_1 \ar[dddl] \ar[ull] \ar[uuull] \ar[dl]  \\
B_2 & & X_3 \ar[rrrdd] \ar[ull] \ar[rrr] \ar[lld] & & & A_3 \ar[d] \ar[uuuuulll] \\
C_2 \ar[u] \ar[rrrdd] & & & & & B_3 \\
& & & & & C_3 \ar[u]\ar[dll] \\
& & & D
}
\end{equation}
The sheaf corresponding to \eqref{eq:big-diagram} belongs to $\cC(\leg^+)$ if and only if the following conditions hold (where $X$ means any $X_i$, etc.):

\begin{enumerate}
\item $U = 0$, 
\item all the maps from $A \to U$, $X \to A$, and $C \to B$ are quasi-isomorphisms, 
\item The total complex of
\begin{equation}
\label{eq:big-micro-stalk}
Z \to \big( X_1 \oplus X_2 \oplus X_3 \big) \to \left(
 \begin{array}{c}
 A_1 \oplus A_2 \oplus A_3 \oplus\\
 C_1 \oplus C_2 \oplus C_3
 \end{array} 
 \right) \to \big( B_1 \oplus B_2 \oplus B_3  \big)\to D
\end{equation}
is acyclic.
\end{enumerate}
It belongs to $\cC(\leg)$ if and only if it obeys the further condition 
\begin{enumerate}
\item[(4)] Each commutative square between $X$ and $D$ is exact, i.e. the total complex of
\[
X \to C_i \oplus C_{i+1} \to D
\]
is acyclic.
\end{enumerate}
Indeed the three commutative $3 \times 3$ squares in \eqref{eq:big-diagram}
\[
\xymatrix{
U & A_2 \ar[r]  \ar[l] & B_2 \\
A_1 \ar[u] \ar[d]& X_1 \ar[l] \ar[r] \ar[u] \ar[d] & C_1 \ar[u] \ar[d] \\
B_1 & \ar[l] C_1 \ar[r] & D
}
\quad
\xymatrix{
U & A_3 \ar[r]  \ar[l] & B_3 \\
A_2 \ar[u] \ar[d]& X_2 \ar[l] \ar[r] \ar[u] \ar[d] & C_2 \ar[u] \ar[d] \\
B_2 & \ar[l] C_2 \ar[r] & D
}
\quad
\xymatrix{
U & A_1 \ar[r]  \ar[l] & B_1 \\
A_3 \ar[u] \ar[d]& X_3 \ar[l] \ar[r] \ar[u] \ar[d] & C_3 \ar[u] \ar[d] \\
B_3 & \ar[l] C_3 \ar[r] & D
}
\]
are each of the form considered in \eqref{eq:typical-TTT}, which establishes (1) and (2).  The equation \eqref{eq:big-micro-stalk} gives the microstalk of the sheaf at the singular point in the vertical codirection.

\subsection{Concrete description of $\cC(\leg)$ and $\cC(\leg^+)$}
\label{subsec:concrete}
Consider the star-shaped quiver with a single sink (let us call it $o$) and $g+3$ sources indexed by the faces $a$ of $\Gamma$.  The derived category of this quiver is equivalent to the category of constructible sheaves on the union of the inner region and the inner sheets (labeled by $D$s and $C$s in the local descriptions \S\ref{subsec:local-study-one-two}--\ref{subsec:local-study-vertex}).  It follows from the discussion in \S\ref{subsec:regular-cell-decomposition} that the restriction functor from $\cC(\leg^+)$ to this union is an equivalence. 

In this quiver description, $\cC(\leg) \subset \cC(\leg^+)$ is a full subcategory, with $F$ belonging to $\cC(\leg)$ if and only if whenever $a$ and $b$ are faces of $\Gamma$ separated by an edge,
\begin{equation}
\label{eq:crossing-condition}
\text{the map $F(a) \oplus F(b) \to F(o)$ is an isomorphism.}
\end{equation}
(In the analogous story for Legendrian curves, this is the ``crossing condition'' of \cite{STZ}.)

\begin{remark}
Note in particular that the quasi-equivalence class of the dg category $\cC(\leg^+)$, since it only depends
on the genus, $g$, does not depend on the graph $\Gamma$.  A more precise form of this statement can be obtained from the picture of \S\ref{subsec:edge-and-disk-moves}:
if $\Gamma'$ is obtained from $\Gamma$ by a sequence of edge moves, then $(\leg')^+$ will be obtained from $(\leg)^+$ by a sequence of disk moves, which induce equivalences of categories as in \cite[\S 2.3]{STW}.  In particular, the moduli spaces of objects in $\cC(\leg)$ and $\cC(\leg')$ will be related to each other by a sequence of cluster transformations.
\end{remark}

\subsection{Moduli of microlocal rank one objects}
\label{def:MGamma}

Note \eqref{eq:crossing-condition} has the consequence that, if $F \in \cC(\leg)$ and $F(o)$ is a vector space concentrated in degree $0$, then $F(o)$ must be even-dimensional, with the spaces labeling the pillows being concentrated in degree $0$ and having half the dimension.  We call that dimension the \emph{microlocal rank} of the object.

Thus if $F$ has microlocal rank one, $F(o)$ is a 2-dimensional vector space.  Let us say that a framed object is one equipped with an isomorphism $F(o) \cong k^2$.  A framing rigidifies an object: there are no automorphisms of an object (nor even self-homotopies of the identity automorphism) that preserve the framing except for the identity --- there is a fine moduli space of framed sheaves of microlocal rank one.  It can be described concretely as an open subset of $(\bP^1)^{ (g+3)}$, where the factors of $\bP^1$ are indexed by the faces of $\Gamma$.

We define an affine open subset $\Mfr(\Gamma) \subset (\bP^1)^{g+3}$ as follows.
A point in $\Mfr(\Gamma)$ is a collection of $z_a\in \bP^1$, one for each face $a$ of $\Gamma$, subject to the condition that $z_a \neq z_b$ whenever $a$ and $b$ share an edge of $\Gamma$.

$\PGL_2$ acts diagonally on $\Mfr(\Gamma)$.  Define $\cM(\Gamma)$ as the quotient $\cM(\Gamma) := \Mfr(\Gamma)/\PGL_2.$ 

\subsection{The chromatic polynomial as a Legendrian invariant}
The definition of $\Mfr(\Gamma)$ makes sense over any commutative ring,
so we consider it over the finite fields $\bF_q$ with $q$ a prime power.  The number of $\bF_q$-points of $\Mfr(\Gamma)$ is equal to the number of $(q+1)$-colorings of the map defined by $\Gamma$, or equivalently the number of $(q+1)$-colorings of the dual planar graph $\hat{\Gamma}$ --- these are the values at $q+1$ of the \emph{chromatic polynomial} of $\hat{\Gamma}$.

\begin{theorem}
\label{thm:gamgamprime}
Suppose $\Gamma$ and $\Gamma'$ are cubic planar graphs with the same number of vertices, so the respective surfaces $\leg$ and $\leg'$ have the same genus and the same classical invariants (cf. Prop. \ref{prop:same-classical-invariants}).
If $\hat{\Gamma}$ and $\hat{\Gamma'}$ do not have the same chromatic polynomial, then $\leg$ and $\leg'$ are not Legendrian isotopic.
\end{theorem}

\begin{proof}
We prove the contrapositive.  Whenever $\leg$ is Legendrian isotopic to $\leg',$ the GKS-equivalence \cite{GKS} gives an isomorphism between $\cM(\Gamma)$ and $\cM(\Gamma')$.
In particular, these moduli spaces have the same number of points over different fields.  By considering $\bF_q$, this means the chromatic polynomials are equal at all prime powers, and thus equal.
\end{proof}

\begin{remark}
\begin{enumerate}
\item
It is natural to ask whether a Legendrian isotopy between $\leg$ and $\leg'$ implies that $\Gamma$ and $\Gamma'$ are equivalent as planar graphs.  Counterexamples have recently been obtained by Roger Casals and Emmy Murphy.  In particular Casals has constructed an infinite family of examples, the simplest of which are the following:
\begin{center}
\begin{tikzpicture}
\draw[thick] (0,0)--(0,2)--(4,1)--(0,0)--(.5,.5)--(.5,1.5)--(0,2)--(.5,1.5)--(1,1.3)--(1,.7)--(.5,.5)--(1,.7)--(1.5,.8)--(1.5,1.2)--(1,1.3)--(1.5,1.2)--(1.5,.8)--(2.7,1)--(1.5,1.2)--(2.7,1)--(4,1);
\end{tikzpicture}
\qquad\qquad
\begin{tikzpicture}
\draw[thick] (0,0)--(0,2)--(3,1)--(0,0)--(.4,.4)--(.6,.8)--(.6,1.2)--(.4,1.6)--(0,2)--(.4,1.6)--(.8,1.3)--(.6,1.2)--(.8,1.3)--(2,1)--(3,1)--(2,1)--(.8,.7)--(.6,.8)--(.8,.7)--(.4,.4);
\end{tikzpicture}
\end{center}
In general Casals's examples are obtained from the blow-up process of \S\ref{subsec:BUF}.  
\item  Dimitroglou Rizell \cite{Rizell} has constructed for each integer $g$ a family of $g+1$ Legendrian embeddings of a genus-$g$ surface into $S^5$, each with $g+1$ Reeb chords, no two of which are Legendrian isotopic.  The hyperelliptic Legendrians have $g+3$ Reeb chords apiece.  We do not know whether there is a literature on the number of different chromatic polynomials of planar graphs, but we suspect that Theorem \ref{thm:gamgamprime} shows that the number of pairwise distinct Legendrian surfaces in our family grows at least exponentially in $g$.
\end{enumerate}
\end{remark}

\subsection{Exact fillings of hyperelliptic Legendrians}

We have discussed a family of singular exact fillings and nonexact fillings in Section \ref{sec:foam}.  Here we can prove the following:

\begin{theorem}
\label{thm:nofillings}
Let $\leg\subset T^{\infty}\bR^3$ be the genus-$g$ Legendrian surface defined by a
simple, cubic planar graph $\Gamma$.  Then $\leg$ has no
smooth oriented graded exact Lagrangian fillings in $\bR^6.$
\end{theorem}

We have given the proof already in the introduction.  Note that it is not strictly necessary for $\Gamma$ to be simple --- the proof works so long as the graph dual to $\Gamma$ has no multiple edges, or even if the number of such edges counted without multiplicity is at least $2g+4$.  We believe the ``graded'' condition can be removed, by using ungraded Floer groups in the construction of \cite{NZ} --- the coefficients of such ungraded groups must have characteristic $2$, but we can still appeal to the same properties of the chromatic polynomial by counting points over the fields $\bF_q$ for $q$ a power of $2$.

Theorem \ref{thm:nofillings} shows there are no smooth exact Lagrangians which define sheaf objects of $\cC(\Gamma).$
Nevertheless, the following proposition shows that the objects still behave cohomologically as though they were genus-$g$ handlebodies.

\begin{proposition}
Suppose $F \in \cC(\Gamma) \subset \Sh(\bR^3)$ has microlocal rank one.  Let us work over a field.  Then
\[
\dim \Ext^i(F,F) = \begin{cases}
1 & \text{if $i = 0$} \\
g & \text{if $i = 1$} \\
0 & \text{otherwise.}
\end{cases}
\]
\end{proposition}
\begin{proof}
By \S\ref{subsec:concrete}, $\cC(\Gamma)$ is a full subcategory of the derived category of a simple quiver, with one sink and $g+3$ sources.   The objects of microlocal rank one are concentrated in homological degree zero, which implies that the Ext groups vanish for $i \notin \{0,1\}$.  As representations of the quiver, the objects of microlocal rank one have dimension $1$ along every source and dimension $2$ along every sink; furthermore each map from a source to a sink is an inclusion, and at least three of these maps have distinct images --- this implies that any endomorphism $F \to F$ must be a scalar.

The $\Ext^1$ calculation follows from the identification of $\Ext^1(F,F)$ with the tangent space to $F$ in $\cM$.  More concretely, a class in $\Ext^1(F,F)$ can be represented by a short exact sequence
\[
0 \to F \to E \to F \to 0
\]
If the stalk of $F$ in the middle region is a two-dimensional vector space $V$ and the stalks in the other regions are lines $L_i \subset V$, then we may assume that the stalk of $E$ in the middle region in $V \oplus V$, the stalk in the other regions is $L_i \oplus L_i$, and that the inclusion maps
$
L_i \oplus L_i \to V \oplus V
$
each have the form
\[
\left(
\begin{array}{rr} 
\mathrm{inc}_{L_i \to V}  & \phi_i \\
0 & \mathrm{inc}_{L_i \to V}
\end{array}
\right)
\]
where the $\phi_i$ are arbitrary linear maps $L_i \to V$.  The data $\{\phi_i\}$ and $\{\phi'_i\}$ represent equivalent extensions if there is a commutative diagram
\[
\xymatrix{
0 \ar[r] \ar[d]& F \ar[r] \ar[d]_{=}& E \ar[r] \ar[d]& F \ar[r]\ar[d]^{=} & 0 \ar[d]\\
0 \ar[r] & F \ar[r] & E' \ar[r] & F \ar[r] & 0 
}
\]
or equivalently, if there is a map $\psi_V:V \to V$ intertwining $\phi_i$ and $\phi'_i$.  In other words, $\Ext^1(F,F)$ is the $g$-dimensional cokernel of the map $\mathrm{End}(V) \to k^{g+3}$.
\end{proof}

\subsection{Period Domain}
\label{subsec:period-domain}

Let $\cP := H^1(\leg;\Gm)$ be the ``period domain'' of the surface $\leg$ --- the name is explained in the next section.  A basis for $H^1(\leg,\bZ)$ gives an identification $\cP \cong (\Gm)^{ 2g}$.  Here we give an alternative description of $\cP$ as a subtorus of $(\Gm)^{ e} = (\Gm)^{(3g+3)}$. 

Each edge $i$ of $\Gamma$ determines a loop on $\leg$, in the following way.  Let $U_i \subset S^2$ be a neighborhood of $i$:
\begin{center}
\includegraphics[scale = .25]{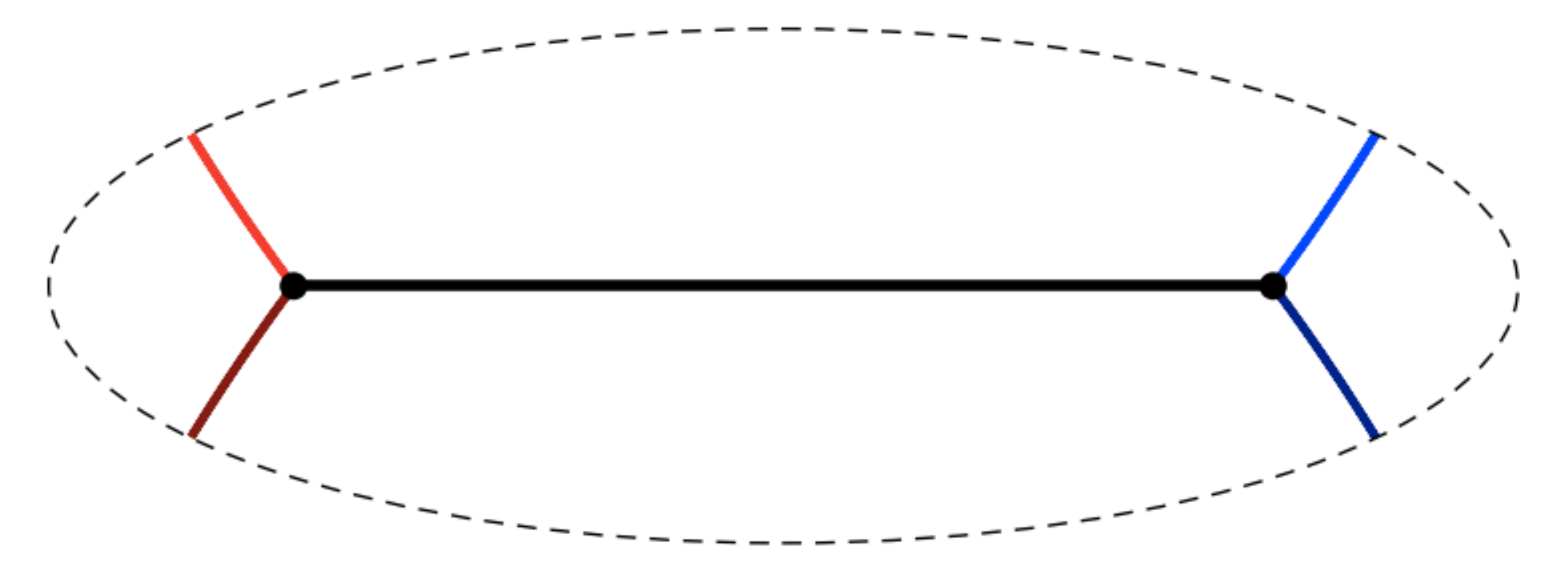} 
\end{center}
The preimage of $U_i$ in $\leg$ is an annulus:
\begin{center}
\includegraphics[scale = .25]{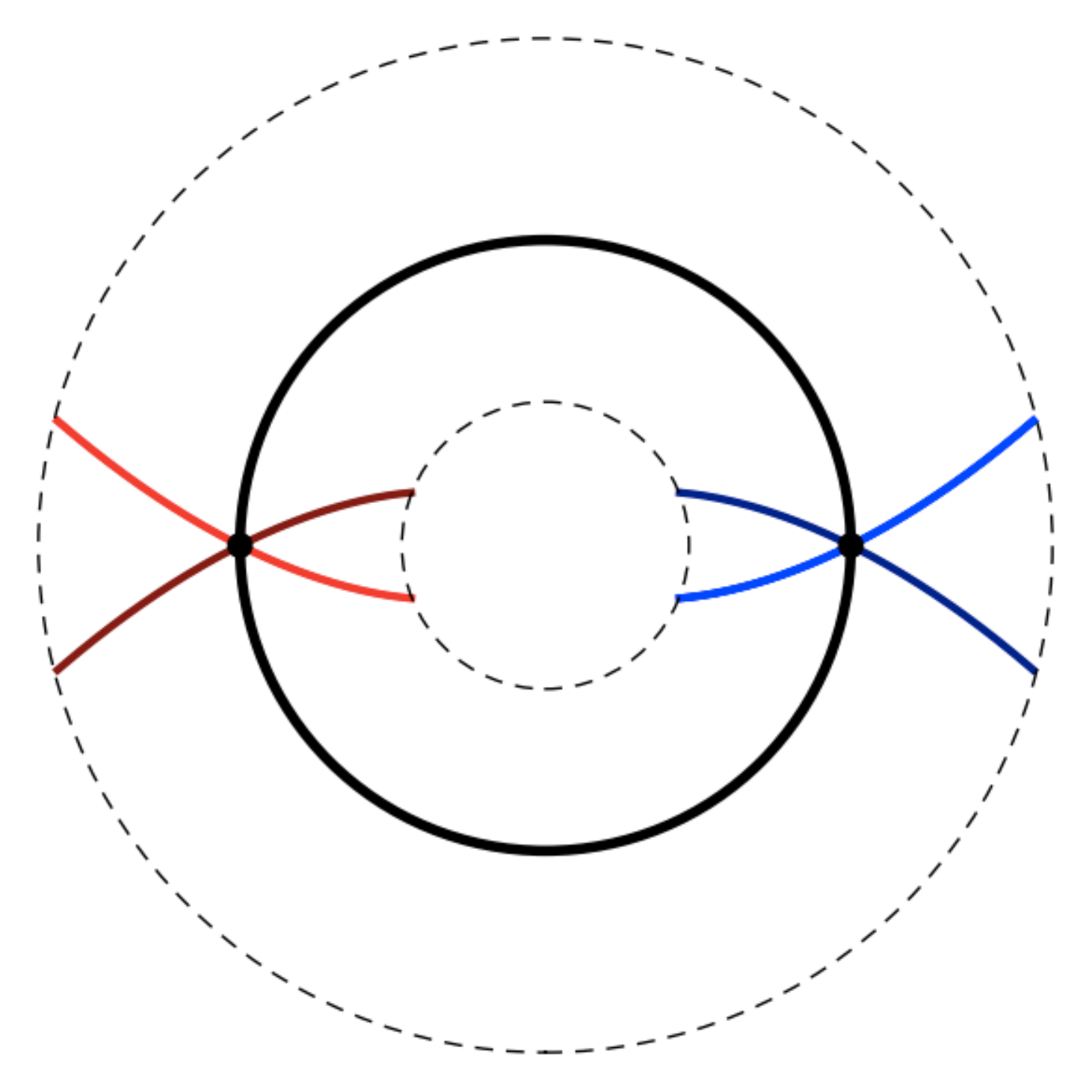}
\end{center}
The preimage of a loop around the edge is disconnected, but the two oriented components represent the same homology class in $H_1(\leg)$, which we denote by $\gamma_i$.  The $\gamma_i$ generate $H_1(\leg)$, with the relations
\begin{equation}
\label{eq:edgerelations}
\sum_{i \in \overline{f}} \gamma_i = 0
\end{equation}
where $\overline{f}$ a closed face, i.e. the closure of a connected component of $S^2 \setminus \Gamma$.
The intersection form $H_1(\leg) \otimes H_1(\leg) \to H_0(\leg) = \bZ$ can be computed from the ribbon structure (i.e. the cyclic ordering on the half-edges incident with a given vertex) on $\Gamma$, as follows:
\begin{equation}
\label{eq:omega-edges}
\langle e_1,e_2\rangle = \begin{cases}
0 & \text{if $e_1$ and $e_2$ are not incident,} \\
\\
1 & \text{if $e_1$ and $e_2$ are incident like \quad
\begin{tikzpicture}\node at (0,0) {$\bullet$}; \draw [thick] (-1,-.4)--(0,0)--(1,-.4); \draw[thick] (0,0)--(0,.8);  \node [above] at (-.6,-.25) {$e_1$}; \node [above] at (.6,-.25) {$e_2$}; \end{tikzpicture}}, \\
\\
-1 & \text{if $e_1$ and $e_2$ are incident like \quad \begin{tikzpicture} \node at (0,0) {$\bullet$}; \draw [thick] (-1,.4)--(0,0)--(1,.4); \draw[thick] (0,0)--(0,-.8); \node [below] at (-.6,.2) {$e_1$}; \node [below] at (.6,.2) {$e_2$}; \end{tikzpicture}} 
\end{cases}
\end{equation}

We regard $\cP$ as an algebraic torus, whose character lattice is $H_1(\leg,\bZ)$.  The intersection form on the character lattice is an element of
$\Hom(\bigwedge^2 H_1(\leg,\bZ),\bZ)$, but being nondegenerate it induces an element in $\bigwedge^2_{\bZ} (H_1(\leg,\bZ))$, which in turn induces a translation-invariant algebraic symplectic form on $\cP$.

\subsection{Period Map}
We parametrize $\Mfr$ in Definition \ref{def:MGamma} as an open subset of $(\bP^1)^{ (g+3)}$ and $\cM$ as the quotient $\Mfr/\PGL_2$.  In \S\ref{subsec:period-domain}, we defined an algebraic torus $\cP$, with a distinguished character $x_e:\cP \to \Gm$ for each edge of $\Gamma$.  We now define a map $\cM \to \cP$, by the formula
\begin{equation}
\label{eq:cross-ratio}
x_i(z \in \Mfr) = -\frac{z_b - z_a}{z_c-z_b}\cdot \frac{z_d-z_c}{z_a-z_d}
\end{equation}
when $a,b,c,d$ are the faces surrounding an edge $i$ in the following pattern:
\begin{center}
\begin{tikzpicture}
\node at (0,.65) {$c$};
\node at (1.75,0) {$b$};
\node at (0,-.65) {$a$};
\node at (-1.75,0) {$d$};
\draw (-2,1)--(-1,0)--(1,0)--(2,1);
\draw (-2,-1)--(-1,0);
\draw (2,-1)--(1,0);
\end{tikzpicture}
\end{center}
One easily verifies the relations $\prod_{e \in \overline{f}} x_e = 1$.  Shen has pointed out to us some additional relations which cut $\cM$ out of $\cP$ as a complete intersection.  Meanwhile there are several ways to see the following:

\begin{proposition}
$\varphi:\cM\to \cP$ is a Lagrangian embedding. 
\end{proposition}
\begin{proof}
As mentioned in Section \ref{sec:priorwork}, an
identical moduli space and Lagrangian embedding was considered in the work of Dimofte-Gabella-Goncharov --- see, e.g., Section 2.3 and Theorem 4.2 of \cite{DGGo}.
\end{proof}

\subsection{Fillings and framings of the period domain}
Let $\cT$ be a tangle obtained by smoothing the $1$-skeleton of a foam, as in \S\ref{subsec:tangles}, and let $L := L(\cT)$ be the Lagrangian 3-manifold associated to $\cT$ in \S\ref{subsec:nonexact-fillings}.  In this section we will assume that $\cT$ has no circle components.  Then the identification of $\leg$ with the boundary of $L$ induces a projection $\pi:H_1(\leg,\bZ) \to H_1(L,\bZ)$ that we will call the \emph{phase} associated to $\cT$.  It is a surjection, and the kernel is identified by
Poincar\'e duality with $H^1(L,\bZ)$.  The phase therefore determines a short exact sequence
\begin{equation}
\label{eq:frame-ses}
\xymatrix{0 \ar[r]& H^1(L,\bZ) \ar[r]& H_1(\leg,\bZ) \ar[r]^{\pi} & H_1(L,\bZ) \ar[r] & 0} 
\end{equation}
The kernel is isotropic with respect to the intersection form on $\leg$. 
We define an \emph{OGW framing}, $\mathfrak{f}:H_1(L,\bZ) \to H_1(\leg,\bZ),$ to be a splitting of this short exact sequence with the same property,
i.e. a map
\begin{equation}
\label{eq:split-to-frame}
\mathfrak{f}: H_1(L,\bZ) \to H_1(\leg,\bZ) \quad \text{with $(\pi \circ \mathfrak{f})(x) = x$ and whose image is isotropic}
\end{equation}
Such an $\mathfrak{f}$ gives a decomposition $H_1(\leg,\bZ) \cong H_1(L,\bZ) \oplus H^1(L,\bZ)$ into dual isotropic subspaces.  We define $T_L := H^1(L,\Gm)$; it is a Lagrangian subspace of the period domain $\cP$ of \S\ref{subsec:period-domain}.
By applying $\Hom(-,\Gm)$ to the splitting \eqref{eq:split-to-frame}, we get a
symplectic covering map (essentially equivalent to the data of a framing $\mathfrak f$, so we reuse the notation)
\[
\mathfrak{f}:T^* (T_L) \to \cP.
\]
When we pull $\cM$ back along $\mathfrak{f}$, it looks
like the graph of a closed one-form whose antiderivative is a multiple-valued holomorphic function $W$ on $T_L$:
\[
\cM = {\rm Graph}(dW),\qquad W: T_L\to \bC.
\] 
In the examples we have checked, it is an integral linear combination of dilogarithms in natural coordinates defined by the phase and framing.
\begin{quote}
\emph{We conjecture that $W$ is the
generating function
for the genus-zero open Gromov-Witten invariants of $L$ in $\bR^6$}
\end{quote}
by analogy with the formulas
of \cite{AV,AKV} for AV branes, proven by Katz and Liu \cite{KL}. 

\begin{remark}
\label{rem:jake-sara}
The open Gromov-Witten invariants have not yet been
defined in this generality, and the role that our framings should play in the theory is still somewhat mysterious ---
but we understand from Jake Solomon that our treatment of framings is in line with expectations.
We further expect that, perhaps after choosing asymptotically radial tangles, the associated Lagrangians
will satisfy the anticipated requisite bounded-geometry requirements (see \cite{GS}) to ensure that the moduli spaces of disks are compact.
Solomon-Tukachinsky \cite[Section 1.2.6]{STdisk} have a project to study open Gromov-Witten theory in this setting.
\end{remark}

\subsection{Examples of OGW framings}

The Lagrangians $L(\cT)$ of \S\ref{subsec:nonexact-fillings} are determined topologically by a tangle in the 3-ball whose endpoints are on the vertices of $\Gamma$. In the diagrams below, we sketch in blue arcs the radial projection of the tangle onto $S^2$.  The brown line segments indicate generators of $H_1(L,\bZ)$ up to sign --- more precisely, the inverse image of each brown line under the double cover $L \to D^3$ is a circle, whose orientation we do not specify.  

\subsubsection{Tetrahedron}
\label{sec:tetrahedronframing}
If $\Gamma$ is the tetrahedron graph, $H_1(\leg,\bZ)$ is generated by $e_1,e_2,e_3$ subject to the relation $e_1 + e_2 + e_3 = 0$. 
Each $e_i$ labels a pair of opposite edges, equivalent under the relations, as in the following diagram.  
\begin{center}
\begin{tikzpicture}
\node at (-3,-9/4) {$\bullet$};
\node at (3,-9/4) {$\bullet$};
\node at (0,0) {$\bullet$};
\node at (0,3) {$\bullet$};
\draw [thick] (-3,-9/4)--(3,-9/4)--(0,3)--(-3,-9/4)--(0,0)--(0,3);
\draw [thick] (3,-9/4)--(0,0);
\node[red, below] at (0,-9/4){$e_1$};
\node[red, left] at (-1.5,1.25/2){$e_3$};
\node[red, right] at (1.5,1.25/2){$e_2$};
\node[red,right] at (0,9/8){$e_1$};
\node[red, below] at (1.25,-1){$e_3$};
\node[red, below] at (-1.25,-1){$e_2$};
\draw [ultra thick, blue, rounded corners] (0,3)--(.7,5/4)--(1.7,-2/4)--(3,-9/4);
\draw [ultra thick, blue, rounded corners] (-3,-9/4)--(-2.3,-6/4)--(-1.3,-3/4)--(0,0);
\draw [ultra thick, brown] (-1.3,-.75)--(-1.12,-.725);
\draw [ultra thick, brown] (-.82,-.71)--(.62,-.59);
\draw [ultra thick, brown] (0.86,-.57)--(1.7,-.5);
\end{tikzpicture}
\end{center}
The intersection form \eqref{eq:omega-edges} is given by $\omega(e_1,e_2) = \omega(e_2,e_3) = \omega(e_3,e_1) =  1$. 
The blue lines indicate a tangle in the interior of the tetrahedron, determining a phase.  In the phase pictured, the loop over $e_2$ (as an element of $H_1(\leg,\bZ)$) maps to zero in $H_1(L,\bZ)$.
More generally, each of the three generators $e_i$ determines a phase, and $H_1(L,\bZ)$ is the quotient of $H_1(\leg,\bZ)$ by $e_i$.

Here is a visualization of the map of lattices $H_1(\leg,\bZ) \to H_1(L,\bZ)$:
\begin{center}
\begin{tikzpicture}
\node at (0,0) {$0$};
\node at (-.5,.866) {$\bullet$};
\node at (.5,.866) {$e_1$};
\node at (0,2*.866) {$\bullet$};
\node at (1,0) {$\bullet$};
\node at (-1,0) {$e_2$};
\node at (-.5,-.866) {$\bullet$};
\node at (.5,-.866) {$e_3$};
\node at (0,-2*.866) {$\bullet$};
\node at (-1.5,-.866) {$\bullet$};
\node at (1.5,-.866) {$\bullet$};
\node at (-1.5,.866) {$\bullet$};
\node at (1.5,.866) {$\bullet$};
\node at (0,-2.5) {$H_1(\leg)$};
\node at (4,0) {$\bullet$};
\node at (4,.866) {$\overline{x}$};
\node at (4,-.866) {$\bullet$};
\node at (4,-2*.866) {$\vdots$};
\node at (4,2*.866) {$\vdots$};
\node at (4,-2.5) {$H_1(L)$};
\draw [thick,->] (2,0)--(3.5,0);
\end{tikzpicture}
\end{center}
The kernel of this map is naturally identified with $H^1(L,\bZ)$ --- in particular $H^1(L,\bZ)$ has a canonical generator, the edge $e_2$ in the picture.  This gives us a preferred orientation for the loop in $L$ that projects to the brown line segment, i.e. a preferred generator for $H_1(L,\bZ)  = \Hom(H^1(L,\bZ),\bZ)$.  In other words we choose the generator $\overline{x}$ for which
\begin{equation}
\label{eq:signconvention}
{\omega(x,e_2) >0}
\end{equation}
holds for any representative $x \in H_1(\leg,\bZ)$ projecting to $\overline{x}$.  The combinatorics of the situation gives us a distinguished choice of $x$, namely ${x = e_1}$.  (This is not typical for more general graphs.)  

When we think of $H^1(L,\bZ)$ as functions (homomorphisms) on $H_1(L,\Gm)$, we will write $e^v$ instead of $e_2$ for the canonical generator.  Similarly when considering it as a function on $H^1(S,\Gm)$, we write $e^u$ instead of $e_1$ for the lift of distinguished lift of the canonical generator of $H_1(L,\bZ)$.  

Any splitting of \eqref{eq:frame-ses}
is Lagrangian, so provides an OGW framing.  The map \eqref{eq:split-to-frame}  carries the canonical generator of $H_1(L,\bZ)$ to ${e_1 + p e_2}$; we write this as $e^{u +pv}$ when we think of it as a function on $\cP$, with $p = 0$ giving the distinguished splitting.  Then $v$ and ${e^{u+pv}}$ define coordinates on $T^*T_L.$

\subsubsection{Triangular prism}
\label{sec:tentframing}
For the triangular prism, the lattice $H_1(\leg,\bZ)$ has rank 4.  
\begin{center}
\begin{tikzpicture}
\node at (-3,-3.2) {$\bullet$};
\node at (3,-3.2) {$\bullet$};
\node at (4/3,-2) {$\bullet$};
\node at (-4/3,-2) {$\bullet$};
\node at (0,0) {$\bullet$};
\node at (0,2) {$\bullet$};
\node[red] at (.2,-2.2){$a$};
\node[red] at (-.5,-1){$b$};
\node[red] at (1,-1){$c$};
\node[red] at (.2,.9){$d$};
\node[red] at (-2,-2.3){$e$};
\node[red] at (2,-2.25){$f$};
\node[red] at (1.3,.2){$h$};
\node[red] at (-1.3,.2){$i$};
\node[red] at (0,-3.5){$g$};
\draw [ultra thick, blue, rounded corners] (-3,-3.2)--(-1,-2.7)--(0,-2.65)--(1,-2.7)--(2,-2.8)--(3,-3.2);
\draw [ultra thick, blue, rounded corners] (4/3,-2)--(.55,-1.2)--(0,0);
\draw [ultra thick, blue, rounded corners] (-4/3,-2)--(0,2);
\draw [ultra thick, brown] (-.75,-.25)--(.55,-1.2);
\draw [ultra thick, brown] (-1,-1)--(0,-2.65);
\draw [thick] (3,-3.2)--(-3,-3.2)--(0,2)--(3,-3.2)--(4/3,-2)--(-4/3,-2)--(0,0)--(0,2);
\draw [thick] (-3,-3.2)--(-4/3,-2);
\draw [thick] (0,0)--(4/3,-2);
\end{tikzpicture}
\end{center}

Inside of $H_1(\leg,\bZ)$ we can find the product of two triangular lattices --- one where $a+b+c = 0$ and one where $g+h+i = 0$.  In these coordinates,
the relations \eqref{eq:edgerelations} imply 
\[
d = a + g \qquad e = c + h \qquad f = b + i
\]
The vectors 
\red{$b,c,h,i$}
make a basis for $H_1(\leg,\bZ)$, and the symplectic form on $\cP$ is given by 
$db \wedge dc + dh \wedge di.$

Let $L$ be the branched double cover of the blue tangle in the diagram above.  We will see that OGW framings of $L$ are 
naturally indexed by $2 \times 2$ symmetric matrices with integer entries.  If we put 
$u_1 = \log(b), u_2 = \log(h)$
and 
$v_1 = \log(c), v_2 = \log(i)$,
then each of these framings gives  an identification of the universal cover of $\cP$ with $T^* \bC^2$, which always has $v_1,v_2$ for momentum coordinates and
$u_s + M_s^t\, v_t$, $s, t = 1, 2$, i.e.
\[
u_1 + M_1^1 v_1 + M_1^2 v_2 \qquad u_2 + M_2^1 v_1 + M_2^2 v_2
\]
for position coordinates.

The two brown arcs lift to two loops in $L$, which is a handlebody of genus two.  Choosing an orientation for each of those loops gives a basis for $H_1(L,\bZ)$.  With respect to this basis for $H_1(L,\bZ)$, and the
$(b,c,h,i)$
basis of $H_1(S,\bZ)$, the projection $H_1(S,\bZ) \to H_1(L,\bZ)$ is the matrix
\begin{equation}
\label{eq:neutron}
\left(
\begin{array}{rrrr}
\epsilon_b & 0 & 0 & 0 \\
0 & 0 & \epsilon_h&  0
\end{array}
\right)
\end{equation}
where $\epsilon_b$ and $\epsilon_h$ are arbitrary signs.  
Let us put $\epsilon_b = \epsilon_h = 1.$

A splitting of the map $H_1(\leg,\bZ) \to H_1(L,\bZ)$ is given by a $4 \times 2$ integer matrix whose product with \eqref{eq:neutron} is the $2 \times 2$ identity matrix, it's general form is
\[
\left(
\begin{array}{rr}
1 & 0 \\
\alpha & \beta \\
0 & 1\\
\gamma & \delta
\end{array}
\right) 
\]
The splitting is Lagrangian if the symplectic pairing of the two columns is zero, i.e. if {$\beta = \gamma$.}
The framing matrix is therefore $M = {\alpha \,\beta \choose \beta \,\delta}.$
The open Gromov-Witten invariants for this example will be discussed in Section \ref{sec:tent}.

\section{Computations, Examples}
\label{sec:exs}

In this section, we choose some graphs $\Gamma$ and compute superpotentials $W: T_L \to \bC$.
We also derive a blow-up formula relating moduli spaces a graph and its blow-up,
but we begin immediately below with the fundamental example:  the tetrahedron.

\subsection{Tetrahedron}
\label{sec:tetrahedron}

In the case of a tetrahedron, we recover the results of \cite{AV,AKV,N-cone}.

We begin by choosing an OGW framing, so we continue with the choice from Section \ref{sec:tetrahedronframing}.
The figure below includes the face coordinates $z_i$ for the framed moduli space.
\begin{center}
\begin{tikzpicture}
\node at (-3,-9/4) {$\bullet$};
\node at (3,-9/4) {$\bullet$};
\node at (0,0) {$\bullet$};
\node at (0,3) {$\bullet$};
\draw [thick] (-3,-9/4)--(3,-9/4)--(0,3)--(-3,-9/4)--(0,0)--(0,3);
\draw [thick] (3,-9/4)--(0,0);
\node[red, below] at (0,-9/4){$e_1$};
\node[red, left] at (-1.5,1.25/2){$e_3$};
\node[red, right] at (1.5,1.25/2){$e_2$};
\node[red,right] at (0,9/8){$e_1$};
\node[red, below] at (1.25,-1){$e_3$};
\node[red, below] at (-1.25,-1){$e_2$};
\draw [ultra thick, blue, rounded corners] (0,3)--(.7,5/4)--(1.7,-2/4)--(3,-9/4);
\draw [ultra thick, blue, rounded corners] (-3,-9/4)--(-2.3,-6/4)--(-1.3,-3/4)--(0,0);
\draw [ultra thick, brown] (-1.3,-.75)--(-1.12,-.725);
\draw [ultra thick, brown] (-.82,-.71)--(.62,-.59);
\draw [ultra thick, brown] (0.86,-.57)--(1.7,-.5);
\node[blue] at (2,1.2){$z_0$};
\node[blue] at (-.8,.4){$z_2$};
\node[blue] at (.7,.4){$z_3$};
\node[blue] at (0,-1.35){$z_1$};
\end{tikzpicture}
\end{center}
A point of $\cM$ is a $\PGL_2$-orbit of quadruples $(z_0,z_1,z_2,z_3)$ which are pairwise distinct.  To compute the image of the period map \eqref{eq:cross-ratio}, we may assume $z_0 = 0, z_1 = 1, z_2 = \infty$, and $z_3 = z$.  Then we find that $\cM$ is parametrized by
\[
(x_1,x_2,x_3) = \left(
\frac{z}{1-z},z-1,\frac{-1}{z}
\right).
\]

The period domain $\cP \subset (\bC^*)^{\times 3}$ is defined by the face relation $x_1 x_2 x_3 = 1,$
and $\cM$ is cut out of $\cP$ by the further equation
${1 + x_2 + x_1 x_2 = 0}$. 
If we identify $\cP$ with $(\bC^*)^2 = \{(x,y)\}$ using coordinates $x = -x_3$ and $y = -x_2x_3$, then  $\cM$ is the pair of pants 
\begin{equation}
\label{eq:popmod}
x + y = 1,
\end{equation}

We take the OGW framing from Section \ref{sec:tetrahedronframing}.
We regard $e^v$ and $e^{u-pv}$ as $\bC^*$-valued functions on $\cP$, with $e^v$ cutting out $T_L$.  
Recall that $e^v = e_2 = x_2$ while ${e^u = e_1 = x_1}$.
The canonical 1-form in this framing is given by $v d(u-pv)$.

Set $$V = -e^{-v}, \qquad U = -(-1)^pe^{u+pv}.$$  
The choice of sign in front is a kind of mirror map --- see Section 6.1 of \cite{AKV}.  It is the $\epsilon_f$ of \eqref{eq:xmasW}.
Write the defining equation as
$1 + \frac{1}{x_2} + x_1 = 0$.
Then  $U$ and $V$ obey
\[
UV^p + V = 1.
\]
The one-form is taken to be $-\log{V}d\log{U},$
a shift of $i\pi d\log{U}$ from $vd\log{U}$ due to the mirror map, for which we have no good explanation.
(Similar sign choices and shifts will be presented without comment in later examples.)

Let us consider the canonical ($p=0$) framing.  Then since we can solve for $V = 1 - U$, we
have for the one-form
\[
-\log(1 - U) U^{-1} dU = \sum_{n = 0}^{\infty} \frac{1}{n} U^{n-1}d U = d\left(\sum_{n = 0}^{\infty} \frac{1}{n^2} U^{n}\right) =  d\Li_2(U).
\]
The conjectural open Gromov-Witten generating function is therefore $W^{(p=0)} = \Li_2(U).$
It obeys the integrality condition, Equation \ref{eq:ovintegrality}.

For a general OGW framing labeled by $p$,
the corresponding open Gromov-Witten invariants have generating function $W^{(p)}$ precisely as in
the work of Aganagic-Klemm-Vafa --- see Section 6.1 of \cite{AKV}.

Many examples and moduli spaces can be related to the tetrahedron case by means of a blow-up formula,
which we now describe.

\subsection{Blow-up Formula}
\label{subsec:BUF}

Let $\Gamma$ be a simple cubic planar graph, $\leg$ the corresponding Legendrian, $\cM$ the moduli space and $\cP$ the period domain.
Let $v\in \Gamma$ be a vertex.  The blow-up of $\Gamma$ at $v$ is a new graph $\Gamma'$ constructed from $\Gamma$ by a local
modification:  $v$ is replaced by a small ``exceptional'' triangle with vertices
connected to the edges incident to $v$ (see picture below).
Let $\Gamma'$ be the blow-up of $\Gamma$ at $v$, and define $\leg',$ $\cM'$ and $\cP'$ respectively.

\begin{center}
\begin{tikzpicture}
\node at (0,0) {$\bullet$};
\node at (.3,.1) {$v$};
\node at (0,5/3) {$\bullet$};
\node at (5/3,-1) {$\bullet$};
\node at (-5/3,-1) {$\bullet$};
\draw[thick] (0,0)--(0,5/3);
\draw[thick] (0,0)--(5/3,-1);
\draw[thick] (0,0)--(-5/3,-1);
\end{tikzpicture}
\qquad\qquad
\begin{tikzpicture}
\node at (0,5/3) {$\bullet$};
\node at (5/3,-1) {$\bullet$};
\node at (-5/3,-1) {$\bullet$};
\node at (0,5/9) {$\bullet$};
\node at (5/9,-1/3) {$\bullet$};
\node at (-5/9,-1/3) {$\bullet$};
\draw[thick] (0,5/9)--(5/9,-1/3);
\draw[thick] (5/9,-1/3)--(-5/9,-1/3);
\draw[thick] (-5/9,-1/3)--(0,5/9);
\draw[thick] (0,5/9)--(0,5/3);
\draw[thick] (5/9,-1/3)--(5/3,-1);
\draw[thick] (-5/9,-1/3)--(-5/3,-1);
\end{tikzpicture}
\vskip 0.1in
Blowing up at a vertex $v$
\end{center}
\vskip 0.1in

\begin{proposition}
\label{prop:blowup}
$\cP'$ has a symplectic decomposition $\cP' \cong \cP \times (\bC^*)^2$
with respect to which $\cM'\subset \cP'$ is a Lagrangian product $\cM\times H$,
where $H\subset (\bC^*)^2$ is the pair of pants $x+y=1.$
\end{proposition}
\begin{proof}
Let $E$ be the set of edges of $\Gamma$ and $E'$ the set of edges of $\Gamma'$.
Note $E\subset E',$ where edges of $\Gamma$ incident to $v$ are mapped to their obvious counterparts (proper transforms) in $\Gamma'.$
Let $N := \bZ\langle E\rangle$ and $N' := \bZ\langle E'\rangle$ be the respective edge lattices, each endowed with their induced intersection forms $A$ and $A'$
from Equation \eqref{eq:omega-edges}.
We define an inclusion $i:N\to N'$ as follows.
For an edge $e\in E$ not incident to $v$, $i(e) = e.$  For $e$ incident to $v$ put $i(e) = e - e',$ where $e'$ is the unique exceptional edge \emph{not}
adjacent to $e$.  The inclusion $i$ induces a map of the same name, $i: N\to N'$.
Put $N_0$ for the lattice generated by the three exceptional edges.
Then a simple case-by-case check shows that
\begin{center}
$N'\cong i(N) \oplus N_0$ is an orthogonal decomposition with respect to $A'$,
and $A'\vert_{i(N)} \cong A$
\end{center}
Now $i^*$ induces a map $j: Hom(E',\bC^*)\to Hom(E,\bC^*)$.  Let $\omega'$ and $\omega$ be the
two-forms on these spaces defined by $A'$ and $A.$  It follows that these forms are related by
$\omega' = j^*\omega + \omega_0$
where $\omega_0$ is $\sum_{i=1}^{3} \frac{dx_i}{x_i}\wedge \frac{dx_j}{x_j},$ the sum taken over the three exceptional divisors.
These $x_i$ obey $\prod_{i=1}^3 x_i = 1$ and as in Equation \ref{eq:popmod} parametrize a pair of pants.
That these pants split as a Cartesian factor follows from the observation (by direct calculation) that for $e \in E$ we have 
$x_e = -x_{i(e)}.$
\end{proof}

\begin{remark}[Tetrahedron, revisited]

The tetrahedron graph is the blow-up of the ``$\Theta$'' graph, the unique (non-simple) planar graph with two vertices and three edges.
The $\Theta$ graph has zero symplectic form and moduli space equal to a point, so Proposition \ref{prop:blowup} establishes that the moduli space for the tetrahedron,
which corresponds to the Aganagic-Vafa brane,
is a pair of pants --- as we have already seen in Section \ref{sec:tetrahedron} above.

\end{remark}

\subsection{Triangular Prism}
\label{sec:tent}
The triangular prism is the blow-up of the tetrahedron graph at any vertex.  We label the edges as in Section \ref{sec:tentframing}.
\begin{figure}[H]
\begin{tikzpicture}
\node at (-3,-3.2) {$\bullet$};
\node at (3,-3.2) {$\bullet$};
\node at (4/3,-2) {$\bullet$};
\node at (-4/3,-2) {$\bullet$};
\node at (0,0) {$\bullet$};
\node at (0,2) {$\bullet$};
\node[red] at (0,-2.2){$a$};
\node[red] at (-0.85,-.97){$b$};
\node[red] at (.95,-1){$c$};
\node[red] at (.2,.9){$d$};
\node[red] at (-2,-2.3){$e$};
\node[red] at (2,-2.25){$f$};
\node[red] at (1.3,.2){$h$};
\node[red] at (-1.3,.2){$i$};
\node[red] at (0,-3.5){$g$};
\node[blue] at (0,-1.3){$x$};
\node[blue] at (2.4,.2){$y$};
\node[blue] at (-1.05,-.4){$z_2$};
\node[blue] at (.9,-.4){$z_3$};
\node[blue] at (0,-2.9){$z_1$};
\draw [ultra thick, blue, rounded corners] (-3,-3.2)--(-1,-2.7)--(0,-2.65)--(1,-2.7)--(2,-2.8)--(3,-3.2);
\draw [ultra thick, blue, rounded corners] (4/3,-2)--(.55,-1.2)--(0,0);
\draw [ultra thick, blue, rounded corners] (-4/3,-2)--(0,2);
\draw [ultra thick, brown] (-.75,-.25)--(.55,-1.2);
\draw [ultra thick, brown] (-1,-1)--(0,-2.65);
\draw [thick] (3,-3.2)--(-3,-3.2)--(0,2)--(3,-3.2)--(4/3,-2)--(-4/3,-2)--(0,0)--(0,2);
\draw [thick] (-3,-3.2)--(-4/3,-2);
\draw [thick] (0,0)--(4/3,-2);
\end{tikzpicture}
\caption{The blue arcs denote the tangle. The brown arcs generate $H_1(L,\bZ).$}
\label{fig:tetframe}
\end{figure}
Using the blow-up procedure, we confirm that the intersection form in the basis
$$\{a, b, c, g, h, i, d - a - g, e - c- h, f - b - i\}$$
is
$$H\oplus H \oplus 0,\qquad H = {\tiny \begin{pmatrix}0&1&-1\\ -1&0&1\\ 1&-1&0\end{pmatrix}}$$

Let's write down the period maps.
$$\begin{array}{ccc}
x_a = -\frac{x-z_2}{z_2-z_1}\cdot \frac{z_1-z_3}{z_3-x}&
x_d = -\frac{z_2-x}{x-z_3}\cdot \frac{z_3-y}{y-z_2}&
x_g = -\frac{z_1-z_2}{z_2-y}\cdot \frac{y-z_3}{z_3-z_1}\\
x_b = -\frac{x-z_3}{z_3-z_2}\cdot \frac{z_2-z_1}{z_1-x}&
x_e = -\frac{z_1-x}{x-z_2}\cdot \frac{z_2-y}{y-z_1}&
x_h = -\frac{z_3-z_1}{z_1-y}\cdot \frac{y-z_2}{z_2-z_3}\\
x_c = -\frac{x-z_1}{z_1-z_3}\cdot \frac{z_3-z_2}{z_2-x}&
x_f = -\frac{z_1-y}{y-z_3}\cdot \frac{z_3-x}{x-z_1}&
x_i = -\frac{z_2-z_3}{z_3-y}\cdot \frac{y-z_1}{z_1-z_2}\\
\end{array}
$$
With the blow-up basis as our guide, we note the following relations.
$$x_a x_b x_c = 1,\qquad 1 + x_c + x_bx_c = 0$$
$$x_g x_h x_i = 1,\qquad 1 + x_h + x_g x_h = 0$$
$$\frac{x_d}{x_a x_g} = -1,\qquad \frac{x_e}{x_c x_h} = -1 \qquad \frac{x_f}{x_b x_i} = -1$$

We see that the intersection form is nondegenerate on $H_1(\leg,\bZ)$ generated by $\{b,c,g,h\}$
and the image of the period map is the Cartesian product in $(\bC^*)^4$ of two pairs of pants.

Let us continue our analysis by picking up from Section \ref{sec:tentframing}, where we considered an OGW phase and
family of framings described by the framing matrix
$M = {\alpha\,\beta\choose\beta\,\delta}$.  We begin by considering ``zero framing," $M=0.$
Proceeding by analogy with Section \ref{sec:tetrahedron},
define the corresponding coordinates as follows.  Put
$$U_1 = {-x_b = -e^{u_1}}, \quad V_1 = -{\frac{1}{x_c} = -e^{-v_1}};
\qquad U_2 = {-x_h = -e^{u_2}},\quad V_2 = {-\frac{1}{x_i} = -e^{-v_2}}.$$
Then
\begin{equation}
\label{eq:twopants}
U_1 + V_1 = 1,\qquad U_2 + V_2 = 1.
\end{equation}
The symplectic form is $\sum_i d\log(V_i) \wedge d\log(U_i).$
The moduli space is Lagrangian and can be written as the graph of $dW(U_1,U_2).$
That is,
we can solve the equation $-\log(V_i) = \partial_{\log(U_i)}W$ for $W.$
To do so, solve the defining equation for $-\log(V_i) = -\log\left(1 - U_i\right),$ which gives
$$W = \Li_2(U_1) + \Li_2(U_2).$$

We now study how $W$ changes for different framings.
Make the change of coordinates $U_1 \to U_1 V_1^\alpha V_2^\beta (-1)^\alpha,$  $U_2 \to U_2V_1^\beta V_2^\delta,$ with
$V_i$ unchanged, so that
Equation \ref{eq:twopants} now reads 
$$U_1V_1^\alpha V_2^\beta (-1)^\alpha + V_1 = 1,\qquad U_2 V_1^\beta V_2^\delta (-1)^\delta + V_2 = 1.$$
We then try to write $-\log(V_i) = \partial_{\log(U_i)}W(U).$  The equations define a new function $W(U)$ for each choice of $M$.
If $M$ is diagonal, $M = {\rm diag}(\alpha,\delta),$ then the equations above decouple and $W = W^{(\alpha)}(x_1) + W^{(\delta)}(x_2)$,
where $W^{(p)}$ is the framing-$p$ superpotential as in Section \ref{sec:tetrahedron}.

Let us investigate some non-diagonal framings, $M$.  Consider the family $M = {0\; p\choose p\; 0}$.
When $p = 1$ we can solve the equations for the $V_i$.
The equations are
$$U_1 V_2 + V_1 = 1,\qquad U_2 V_1 + V_2 = 1.$$
We can solve
$$-\log{V_1} =  -\log(1-U_1) +\log(1-U_1U_2) ,\qquad -\log(V_2) = -\log(1-U_2) + \log(1-U_1U_2).$$
Putting $-\log(V_i) = \partial_{\log(U_i)}W(U)$ gives
$$W = \Li_2(U_1) + \Li_2(U_2) - \Li_2(U_1U_2),$$
an integral linear combination of dilagarithms with arguments labeled by $H_1(L,\bZ)$, as expected --- see Equation \eqref{eq:ovintegrality}.
When $p = -1,$ the equations for the $y_i$ are quadratic, and an exact solution for $W$ seems out of reach.
We can instead develop a power series solution $W = \sum_{a,b} K_{a,b} U_1^a U_2^b$, solve for the conjectural
open Gromov-Witten invariants $K_{a,b}$.  We then want to check that they define integer conjectural BPS numbers after
accounting for $d$-fold covers with the $1/d^2$ multiple-cover formula of Ooguri-Vafa
--- equivalently, check that to a specified order $W$ has the form
$$W = \sum_{d = (d_1,d_2)} a(d) \,\Li_2(U_1^{d_1} U_2^{d_2})$$
with $a(d)$ integers, where $d_1$ and $d_2$ refer to homology classes corresponding to the upper and lower brown arcs of Figure \ref{fig:tetframe},
respectively (and we have suppressed the dependence on framing in the notation $a(d)$).
We find for $a(d)$ the following numbers.

\vskip0.1in
\begin{center}
\begin{tabular}{|c||c|c|c|c|c|c|c|c|c|c|}
\hline
$d_1\; \backslash \;d_2$&0&1&2&3&4&5&6&7&8&9\\
\hline
\hline
0&0&1&0&0&0&0&0&0&0&0\\
1&1&1&1&1&1&1&1&1&1&\\
2&0&1&2&4&6&9&12&16&&\\
3&0&1&4&11&25&49&87&&&\\
4&0&1&6&25&76&196&&&&\\
5&0&1&9&49&196&&&&&\\
6&0&1&12&87&&&&&&\\
7&0&1&16&&&&&&&\\
8&0&1&&&&&&&&\\
9&0&&&&&&&&&\\
\hline
\end{tabular}
\end{center}
\vskip0.1in

We have written a computer code to implement this procedure, and integrality has been verified in
all of the hundreds of examples checked.

\subsection{The Cube}

The $1$-skeleton of a cube is not obtained from the ``blow-up'' construction, and 
the moduli space is not a product of pairs of pants, so presents an interesting new test of our methods.

\begin{center}
\begin{tikzpicture}
\node at (-1,-1) {$\bullet$};
\node at (1,-1) {$\bullet$};
\node at (1,1) {$\bullet$};
\node at (-1,1) {$\bullet$};
\node at (-2,-2) {$\bullet$};
\node at (2,-2) {$\bullet$};
\node at (2,2) {$\bullet$};
\node at (-2,2) {$\bullet$};
\draw [thick] (-2,-2)--(2,-2)--(2,2)--(-2,2)--(-2,-2)--(-1,-1)--(1,-1)--(1,1)--(-1,1)--(-1,-1);
\draw [thick] (2,-2)--(1,-1);
\draw [thick] (2,2)--(1,1);
\draw [thick] (-2,2)--(-1,1);
\node[red] at (0,-2){$5$};
\node[red] at (2,0){$6$};
\node[red] at (0,2.05){$7$};
\node[red] at (-2,0){$8$};
\node[red] at (0,-1){$1$};
\node[red] at (1,0){$2$};
\node[red] at (0,1){$3$};
\node[red] at (-1,0){$4$};
\node[red] at (-1.5,--1.5){$12$};
\node[red] at (1.5,-1.5){$10$};
\node[red] at (1.5,1.5){$11$};
\node[red] at (-1.5,-1.5){$9$};
\node[blue] at (0,-1.5){$w$};
\node[blue] at (0,1.5){$y$};
\node[blue] at (-1.5,0){$z$};
\node[blue] at (1.5,0){$x$};
\node[blue] at (0,0){$u$};
\node[blue] at (2.7,1.2){$v$};
\draw [ultra thick,blue,rounded corners] (2,-2)--(1.8,-1)--(1.8,1)--(2,2);
\draw [ultra thick,blue,rounded corners] (-2,-2)--(-1.9,-1.6)--(-1.4,-1.1)--(-1,-1);
\draw [ultra thick,blue,rounded corners] (-1,1)--(-.6,1.2)--(.6,1.2)--(1,1);
\draw [ultra thick,blue,rounded corners] (-2,2)--(-1.9,1.6)--(-1.1,.6);
\draw [ultra thick,blue,rounded corners] (-.9,.35)--(-.344,-.344)--(.7,-.9)--(1,-1);
\draw [ultra thick,brown,rounded corners] (-1.4,-1.1)--(-1.3,-.8)--(-1.1,-.7);
\draw [ultra thick,brown,rounded corners] (-.9,-.6)--(-.344,-.344);
\draw [ultra thick,brown,rounded corners] (.7,-.9)--(.6,.15)--(.6,.93);
\draw [ultra thick,brown,rounded corners] (.6,1.05)--(.6,1.2);
\draw [ultra thick,brown,rounded corners] (-1.9,1.6)--(-1.8,1.7);
\draw [ultra thick,brown,rounded corners] (-1.7,1.8)--(-1.5,1.85);
\draw [ultra thick,brown,rounded corners] (-1.5,1.85)--(1.5,1.85)--(1.66,1.77);
\draw [ultra thick,brown,rounded corners] (1.8,1.7)--(1.92,1.6);
\end{tikzpicture}
\end{center}

We choose the following basis for $H^1(\leg,\bZ)$
$$\{-(e_1 + e_3), e_9; -e_2, e_3; -e_7,e_6\},$$
in which the symplectic form is standard.
The phase is evident from the blue tangle, which determines the kernel of the map $H_1(\leg,\bZ)\to H_1(L,\bZ)$
to be $e_9 =: e^{v_1}, e_3 =: e^{v_2}, e_6 = : e^{v_3}.$
In choosing this basis we have also selected a lift of the brown generators of $H_1(L,\bZ),$ namely $e_1+e_3 =: e^{u_1},
e_2 =: e^{u_2}, e_7 =: e^{u_3}.$
We must try to express the $v_i$ in terms of the $u_i$.

To do so, calculate the monodromy from Equation \eqref{eq:cross-ratio}:
$$\begin{array}{ccc}
x_1 = -\frac{u-z}{z-w}\cdot \frac{w-x}{x-u}&
x_2 = -\frac{u-w}{w-x}\cdot \frac{x-y}{y-u}&
x_3 = -\frac{u-x}{x-y}\cdot \frac{y-z}{z-u}\\
x_4 = -\frac{u-y}{y-z}\cdot \frac{z-w}{w-u}&
x_5 = -\frac{v-x}{x-w}\cdot \frac{w-z}{z-v}&
x_6 = -\frac{v-y}{y-x}\cdot \frac{x-w}{w-v}\\
x_7 = -\frac{v-z}{z-y}\cdot \frac{y-x}{x-v}&
x_8 = -\frac{v-w}{w-z}\cdot \frac{z-y}{y-v}&
x_9 = -\frac{w-u}{u-z}\cdot \frac{z-v}{v-w}\\
x_{10} = -\frac{x-u}{u-w}\cdot \frac{w-v}{v-x}&
x_{11} = -\frac{y-u}{u-x}\cdot \frac{x-v}{v-y}&
x_{12} = -\frac{z-u}{u-y}\cdot \frac{y-v}{v-z}\\
\end{array}$$

We have the face relations:
$$\begin{array}{ccc}
x_1 x_2 x_3 x_4 = 1&x_5 x_6 x_7 x_8 = 1\\
x_1 x_9 x_5 x_{10} = 1&x_2 x_{10} x_6 x_{11} = 1\\
x_3 x_{11} x_7 x_{12} = 1& x_4 x_{12} x_8 x_9 = 1\\
\end{array}
$$
We can use the last 3 to eliminate $x_{12}, x_{11}, x_{10}$ in terms of the others,
and use the first two to eliminate $x_4, x_5.$
$$x_4 = \frac{1}{x_1 x_2 x_3},\qquad x_5 = \frac{1}{x_6 x_7 x_6}$$
$$x_{12} = \frac{x_1 x_2 x_3}{x_8 x_9}\qquad x_{11} = \frac{x_8 x_9}{x_1 x_2 x_3^2 x_7},\qquad x_{10} = \frac{x_1 x_3^2 x_7}{x_6 x_8 x_9}.$$
The last relation gives
$$x_8 = \pm \frac{x_1x_3}{x_6}$$
So the codomain of the period map is coordinatized by 
$$\begin{array}{ccc}
x_9 = -\frac{w-u}{u-z}\cdot \frac{z-v}{v-w}&
x_1 = -\frac{u-z}{z-w}\cdot \frac{w-x}{x-u}&
x_2 = -\frac{u-w}{w-x}\cdot \frac{x-y}{y-u}\\
x_3 = -\frac{u-x}{x-y}\cdot \frac{y-z}{z-u}&
x_6 = -\frac{v-y}{y-x}\cdot \frac{x-w}{w-v}&
x_7 = -\frac{v-z}{z-y}\cdot \frac{y-x}{x-v}\\
\end{array}$$
and we can describe the image, the chromatic Lagrangian, by finding relations among combinations of these variables.

Here's one:
$$x_6\left( 1 + \frac{1}{x_1x_3x_7}\right) + \frac{1}{x_7} + 1 = 0$$
We can eliminate
$x_6 = -\frac{x_1x_3x_7 + x_1x_3}{x_1x_3x_7 + 1}.$

In looking for relations, we will set $w = 0, x = 1, u = \infty.$  Then
$$x_9 = -\frac{v-z}{v},\quad x_1 = -\frac{1}{z},\quad x_2 = y-1,$$
$$x_3 = -\frac{y-z}{y-1},\quad x_6 = -\frac{y-v}{y-1}\cdot \frac{1}{v},\quad
x_7 = -\frac{v-z}{z-y}\cdot \frac{y-1}{1-v}$$
Corresponding to our symplectic basis, define
$$a_1 = (x_1 x_3)^{-1},\quad b_1 = x_9,\qquad a_2 = (x_2)^{-1},\quad b_2 = x_3,\qquad a_3 = x_7^{-1},\quad b_3 = x_6.$$
Then $$\omega = \sum_{i=1}^3 \frac{1}{a_ib_i}da_i\wedge db_i =  \sum_{i=1}^3  \frac{1}{U_iV_i}dU_i\wedge dV_i$$
where we have defined mirror coordinates $U_i$ and $V_i$ by the following choice of signs:
$$U_i = -a_i^{-1},\qquad V_i = -b_i^{-1}.$$
We can solve for $y,z,v$ in terms of $U_i$ and use the results to express the $V_i$ as functions of the $U_i$.
$$y = 1+U_2,\qquad z = \frac{1+U_2}{1-U_1U_2},\qquad v = \frac{-(U_1U_2+U_1)U_3+U_2+1}{-(U_1U_2+U_1)U_3-U_1U_2+1}$$
from which we get
$$b_1 = \frac{U_1U_2U_3(U_1+1)}{(U_1U_2-1)(U_1U_3-1)},\quad b_2 = \frac{U_1(U_2+1)}{U_1U_2-1},\quad b_3 = \frac{U_1(U_3+1)}{U_1U_3-1}.$$
(As an example, the expression for $b_3$ is equivalent to the relation found above involving $x_6 = b_3.$)
Lifting to the universal cover, where all Lagrangians are exact, we define
$$u_i = -\log{U_i},\qquad v_i = -\log{V_i}$$
so 
$$\omega = \sum_i du_i \wedge dv_i = -d \sum v_i d u_i.$$
Then by above, the $v_i$ are functions of the $u_i$ and Lagrangianicity can be verified:  $$\partial_{u_j} v_i = \partial_{u_i}v_j.$$
We now ask Mathematica determine a function
$W(U)$ such that $v_i = \partial_{\log U_i}W.$
We find the solution

$$W = \text{Li}_2\left(U_1\right)+\text{Li}_2\left(U_2\right)+\text{Li}_2\left(U_3\right)-\text{Li}_2\left(U_1U_2\right)-\text{Li}_2\left(U_1U_3\right)$$

Since $W$ is expressed purely in terms of dilogarithms of monomials in the $U_i$,
it satisfies the open Gromov-Witten integrality constraint, Equation \eqref{eq:ovintegrality}. 
The result predicts the existence of unique holomorphic disks
in $\bC^3$ bounding the smooth, non-exact Lagrangian $L$ in various homology classes labeled by the $u_i.$
In particular, we expect the following BPS numbers for this OGW framing:
\medskip
\begin{center}
\fbox{$a{(1,0,0)} = a{(0,1,0)} = a{(0,0,1)} = 1,\qquad a{(1,1,0)} = a{(1,0,1)} = -1$}
\end{center}
\medskip

\section*{Appendix:  Physical Contexts}
\label{app:physics}

There is a wide array of physical set-ups where the mathematics of the present paper applies.  Here is a partial list.

\begin{itemize}
\setlength{\itemsep}{5pt}
\vskip 0.1in
\item {\bf Type-IIA on Noncompact Calabi-Yau.}  Type-IIA string theory on a noncompact Calabi-Yau manifold $X$ has
an effective four-dimensional supergravity theory with $\cN = 2$ supersymmetry and $b_2(X)+1$ abelian gauge fields
arising from the Ramond-Ramond sector.  They are part of chiral
superfields.  (In our set-up, the Calabi-Yau manifold is $\bC^3$, so there is more supersymmetry,
but it may be a good idea to think of $\bC^3$ as a special case of the more general set-up.)
Different couplings of these gauge fields are described by topological string amplitudes at genus $g$.  

A D4-brane whose (five-dimensional) world volume fills a two-plane in spactime cross a supersymmetric
Lagrangian three-cycle $L\subset X$
creates a BPS domain wall.  These domain walls have a net effect on the 4d physics.
The corresponding term in the 4d action, which includes a delta function supported on the domain wall,
includes the contribution
of D2-D0 branes ending on the D4-brane.  These can be computed
by open topological string amplitudes:  open Gromov-Witten invariants counting holomorphic
maps from a disk with boundary lying in $L$.  As we recall below, Ooguri-Vafa derived integrality results for open Gromov-Witten theory
by comparison of this set-up with M-Theory, analagous to the Gopakumar-Vafa formula in the closed case.

References: \cite{AGNT, OV}

\item {\bf Type II-A, Part 2}

One could instead consider a D6-brane wrapping all four dimensions of spacetime cross $L\subset X.$
In this case, the open Gromov-Witten invariants contribute couplings involving
chiral fields, $b_1(L)$ in number, the $b_2(X)+1$ closed chiral fields discussed above, and the graviphoton
multiplet.  One such term is the superpotential, and it is computed from disk invariants \cite{OV}.

\item {\bf M-Theory on a $G_2$-holonomy manifold.}

The M-Theory equivalent of IIA on a Calabi-Yau manifold $X$ is M-Theory on $Y =  X\times S^1,$ where the radius of the circle is
related to the string coupling constant.  To model the D4-brane set-up, one takes an M5-brane wrapping $\bR^2\times L\times S^1.$
Contributions from bound D2-D0 branes are encoded by $M2$-branes with boundary on the M5-brane.
The non-perturbative M-theory set-up allows one to calculate these contributions without resorting to the
perturbative methods of Gromov-Witten theory.  Using this perspective, Ooguri-Vafa determine strict integrality conjectures for
open Gromov-Witten theory, analogous to Gopakumar-Vafa formulas in the closed case.  As a result,
the superpotential $W$ of the four-dimensional theory must have an expression as follows.
First choose a (non-canonical) splitting $H_2(X,L;\bZ)\cong H_2(X;\bZ)\oplus H_1(L,\bZ)$ along with bases $\{C_j\}$ for $H_2(X;\bZ)$ and
$\{\gamma_i\}$ for $H_1(L;\bZ)$.  Then introduce corresponding coordinates $q_j, x_i \in \bC^*$; and for $\beta = \sum_j \beta_j C_j
\in H_2(X;\bZ)$ and $\gamma = \sum_i d_i \gamma_i
\in H_1(L;\bZ),$
define $q^\beta := \prod_j q_j^{\beta_j}$ and $x^\gamma := \prod_i x_i^{d_i}.$  Then
$$\phantom{2}\qquad W(t,x_i) = \sum_{n=1}^\infty \sum_{\beta\in H_2(X)}\sum_{\gamma\in H_1(L)}\sum_{s\in \frac{1}{2}\bZ_+} n_{\beta,\gamma,s}\frac{1}{n^2} q^{n\beta} x^{n\gamma}u^s,\qquad n_{\beta,\gamma,s}\in \bZ$$
A special case arises when, as for our examples, $H_2(X)=0$ so the $q$ term disappears (or one can take the $q\to 1$ limit) and we ignore
the spin dependence by setting $u\to 1$ as well.  We then 
have
\begin{equation}
\label{eq:ovintegrality}
{\phantom{2}}\qquad W(x) = \sum_{\gamma \in H_1(L;\bZ)} c_\gamma \Li_2(x^\gamma),\qquad c_\gamma \in \bZ,
\end{equation}
where $\Li_2(x)$ is the dilogarithm function $\sum_{n\geq 1}\frac{x^n}{n^2}.$  The integrality constraint to which we subject our
conjectural superpotentials is that $W$ can be written in this form with the $c_\gamma$ integers.

\item {\bf M-Theory, Part 2}

The M-Theory equivalent of IIA on $X$ with a D6-brane wrapping four-dimensional spacetime cross a Lagrangian $L\subset X$
is entirely geometrical.  There are no branes, but rather, the compactification manifold is a $G_2$-holonomy seven-fold $Y$,
and should take the form
of a singular $S^1$-fibration, where the circle fiber degenerates to a point over the locus $L.$  A local model should be
the Taub-NUT geometry in the four dimensions transverse to (the lift of) $L$ in $Y$.
The effective theory on the 4d spacetime has $\cN=1$ supersymmetry.  There are as many chiral superfields as there
are $L^2$-harmonic three-forms.  (This number equals $b_4(Y) = b_3(Y)$ if $Y$ is compact, but that is not the case here.)
This picture is largely conjectural, but for the case of the (singular) Harvey-Lawson brane it is rigorous, as shown by Atiyah and Witten.

In this set-up, the superpotential for the chiral superfields of the four-dimensional theory is generated by M2-branes.  There are no M5-branes present,
so the M2-branes are closed.  As explained at the end of the Section 3.1 of \cite{AKV}, for example, the superpotential term receives contributions
from M2-instantons which are homology three-spheres.  The three sphere geometry comes from a disk cross the M-theory circle, with the fibers
over the boundary circle of the disk
identified, since the M-theory circle collapses there.

These ideas are explored in the following works:  \cite{AKV, AMV, AVflop, AW}.

\item{\bf Dimensional reduction of six-dimensional theories}

As discussed in Section \ref{sec:priorwork}, the prior works \cite{CCV, CNV, CEHRV} as well as \cite{DGGo, DGGu, DGH}
employ overlapping mathematical machinery to the present work.  

The works of Section \ref{sec:vafaetal} consider the effective theory of M5-branes wrapping spacetime crossed with a branched double cover of $\bR^3$.
The authors arrive at dual three-dimensional theories by considering different Seifert surfaces bounding the branch locus, a tangle.
The effective three-dimensional theory is an abelian $\cN=2$ Chern-Simons theory coupled to chiral matter,
whose partition function can be computed by supersymmetric localization.
The authors also give an interpretation in terms of a quantum-mechanical wavefunction that bears a strong resemblance to the superpotential $W$
computed here.
A main point of these works is that
different Seifert surfaces generate equivalent Lagrangian descriptions of a single quantum theory.
Invariance of the three-dimensional theory amounts to quantum dilogarithm identities, giving a physical explanation of the
Kontsevich-Soibelman wall-crossing formulas.

The works of Section \ref{sec:dimofteetal} propose various dualities by considering dimensional reductions 
of supersymmetric theories on brane worldvolumes.
For example, one can consider a stack of $n$ M5-branes wrapping $\bR^2\times S^1 \times L$, where $L$ is
a (typically noncompact) three-dimensional Lagrangian in a (typically noncompact) Calabi-Yau three-fold $X$. 
Boundary conditions on $L$ are
fixed:  some prototypes are when $L$ is a knot complement for a hyperbolic knot $K\subset S^3$
or a Harvey-Lawson/Aganagic-Vafa brane in a toric Calabi-Yau threefold.  By relating the six-dimensional theory on
the brane to the reduction to \emph{either} $\bR^2 \times S^1$ or to the three-fold $L,$ a duality is proposed in \cite{DGH, DGGu} between
a three-dimensional $\cN = 2$ supersymmetric theory on $\bR^2 \times S^1$ and an $SL(n,\bC)$ Chern-Simons (CS) theory on $L$.
The CS partition function has an expression in terms of dilogarithms arising as the hyperbolic volumes of ideal tetrahedra.
The partition function of the $\cN=2$ theory is an index, counting BPS states, and involves dilogarithms, as we relate below.

[The earlier work of \cite{DGH} considered a further reduction on $S^1$ to a two-dimensional $\cN = (2,2)$ supersymmetric theory on $\bR^2$.
In the reduction to two dimensions, the partition function is expressed as a sum over vortex instantons on $\bC$, and
the vortex center locations are described by symmetric polynomials.  The fixed-point theorem relates
the character over this moduli space to monomial contributions, leading to the dilogarithm expressions above.]

As described in \cite{DGH}, all this can be seen from the spacetime perspective by decoupling BPS states from the
bulk by setting the K\"ahler parameters of the Calabi-Yau to large volume, similar to the $H_2(X,\bZ)=0$ case discussed above.

\end{itemize}

\end{document}